\documentclass[11pt]{amsart}

\pdfoutput=1

\usepackage{amssymb,amsfonts,amsmath,amsthm}
\usepackage{url}
\usepackage{color}
\usepackage{enumitem}
\usepackage{epsfig,graphicx,psfrag}
\usepackage{tikz-cd}
\usepackage{comment}
\usepackage{pinlabel}

\usepackage[top=1.2in,bottom=1.4in,left=1.4in,right=1.4in]{geometry}

\newtheorem{theorem}{Theorem}
\newtheorem{proposition}[theorem]{Proposition}
\newtheorem{corollary}[theorem]{Corollary}
\newtheorem{lemma}[theorem]{Lemma}

\newtheorem{convention}[theorem]{Convention}

\newtheorem{question}[theorem]{Question}
\newtheorem{problem}[theorem]{Problem}

\theoremstyle{definition}
\newtheorem{definition}[theorem]{Definition}
\newtheorem{remark}[theorem]{Remark}

\newtheorem{notation}[theorem]{Notation}
\newtheorem{example}[theorem]{Example}

\newtheorem*{remark*}{Remark}

\newcommand{\C}{\mathbb{C}}
\newcommand{\R}{\mathbb{R}}

\newcommand{\Z}{\mathbb{Z}}
\newcommand{\N}{\mathbb{N}}
\renewcommand{\H}{\mathbb{H}}

\newcommand{\PSL}{\mathrm{PSL}}

\newcommand{\Hom}{\mathrm{Hom}}

\newcommand{\id}{\mathrm{id}}

\newcommand{\F}{\mathcal{F}}

\newcommand{\UTtM}{\mathrm{UT}\tilde{M}}
\newcommand{\Fqg}{\mathcal{F}^{QG}}
\newcommand{\UT}{\mathrm{UT}}

\DeclareMathOperator{\Isom}{Isom}
\DeclareMathOperator{\Homeo}{Homeo}
\DeclareMathOperator{\Diff}{Diff}

\renewcommand{\tilde}{\widetilde}
\renewcommand{\paragraph}[1]{\medskip \noindent \textbf{#1. }}

\newcounter{notes}

\title{$C^0$ stability of boundary actions and inequivalent Anosov flows}

\author{Jonathan Bowden}
\address{Mathematische Fakult\"at 
Universit\"at Regensburg
93053 Regensburg, Germany}
\email{jonathan.bowden@math.uni-regensburg.de}
\author{Kathryn Mann}
\address{Department of Mathematics, Cornell University,  Ithaca, NY 14853, USA}
\email{k.mann@cornell.edu}

\begin{document}

\maketitle
\numberwithin{theorem}{section} 
\begin{abstract}
We give a topological stability result for the action of the fundamental group of a compact manifold of negative curvature on its boundary at infinity: any nearby action of this group by homeomorphisms of the sphere is semi-conjugate to the standard boundary action.   Using similar techniques we prove a global rigidity result for the ``slithering actions" of 3-manifold groups that come from skew-Anosov flows.  As applications, we construct hyperbolic 3-manifolds that admit arbitrarily many topologically inequivalent Anosov flows, answering a question from Kirby's problem list, and also give a more conceptual proof of a theorem of the second author on {\em global} $C^0$--rigidity of geometric surface group actions on the circle.  
\end{abstract}

\section{Introduction}

This paper proves two related rigidity results for group actions on manifolds, with applications to skew-Anosov flows. The first is a general local rigidity result for the {\em boundary action} of the fundamental group of a closed negatively curved manifold.

\paragraph{1. Local rigidity of boundary actions} 
A major historical motivation for the study of rigidity of group actions comes from the classical (Selberg--Calabi--Weil) rigidity of lattices in Lie groups.  Perhaps the best known example is Calabi's original theorem that, for $n \geq 3$, the fundamental group of a compact, hyperbolic $n$-manifold is locally rigid as a lattice in $\mathrm{SO}(n,1)$, later extended to a global rigidity result by Mostow.   From a geometric-topological viewpoint, it is natural to consider the action of $\mathrm{SO}(n,1)$ on the boundary sphere of the compactification of hyperbolic $n$-space (the universal cover of the manifold in question) and several modern proofs of Mostow rigidity pass through the study of this boundary action.  See \cite{Fisher} for a broad introduction to the subject.  

More generally, if $M$ is a closed $n$-dimensional manifold of (variable) negative curvature, its universal cover $\tilde{M}$ still admits a natural compactification by a visual {\em boundary sphere}, denoted $\partial_\infty \tilde{M}$ and the action of $\pi_1M$ on $\tilde{M}$ by deck transformations extends to an action by homeomorphisms on $\partial_\infty \tilde{M}$, which we call the {\em boundary action}.   However, even if $M$ is smooth, $\partial_\infty \tilde{M}$ typically has no more than a $C^0$ structure.  This presents a new challenge for dynamicists, as many tools in rigidity theory originate either from hyperbolic smooth dynamics or the homogeneous (Lie group) setting, where differentiability plays an essential role. 
 
As we will later show, in the $C^0$ context the best rigidity result one can hope for is {\em topological stability}. 
An action $\rho': \Gamma \to \Homeo(X)$ of a group $\Gamma$ on a space $X$ is said to be a {\em topological factor} of an action $\rho: \Gamma \to \Homeo(Y)$ if there is a surjective, continuous map $h: X \to Y$ (called a semiconjguacy) such that $h \circ \rho' = \rho \circ h$.  A group action $\rho: \Gamma \to \Homeo(X)$ is {\em topologically stable} if any action which is close to $\rho$ in $\Hom(\Gamma, \Homeo(Y))$ is a factor of $\rho$.\footnote{Elsewhere in the literature this is also referred to as {\em semi-stability} or {\em (topological) structural stability}, see \cite{Nit} for some discussion on terminology.}   Here and in what follows, $\Hom(\Gamma, \Homeo(Y))$ is always equipped with the standard compact-open topology.   Our first result is the following.

\begin{theorem}[Topological stability of boundary actions] \label{thm:local_rig}
Let $M$ be a compact, orientable $n$-manifold with negative curvature, and $\rho_0: \pi_1M \to \Homeo(S^{n-1})$ the natural boundary action on $\partial_\infty \tilde{M}$. 
There exists a neighborhood of $\rho_0$ in $\Hom(\pi_1M, \Homeo(S^{n-1}))$ consisting of representations which are topological factors of $\rho_0$.   

Moreover, this topological stability is {\em strong} in the following sense: for any neighborhood $U$ of the identity in the space of continuous self-maps of $S^{n-1}$, there exists a neighborhood $V$ of $\rho_0$ in $\Hom(\pi_1M, \Homeo(S^{n-1}))$ so that every element of $V$ is semi-conjugate to $\rho_0$ by some map in $U$.  
\end{theorem} 

The statement of Theorem \ref{thm:local_rig} is similar in spirit to the extensions of the classical $C^1$-structural stability for Anosov (or more generally, Axiom A) systems to topological stability proved by Walters \cite{Walt} and Nitecki \cite{Nit} in the 1970s.  However, we are working in the context of group actions rather than individual diffeomorphisms, and further, we do not assume any regularity of the original boundary action that is to be perturbed.  Thus, our tools are by necessity fundamentally different.  

In the context of stability properties of group actions Sullivan \cite{Sul} characterized which subgroups of $\PSL(2,\C)$ exhibit $C^1$-stability\footnote{ Sullivan's main result is in the opposite direction to ours: he shows that, among subgroups of $\PSL(2,\C)$, $C^1$-stability implies convex-cocompactness.}  and also remarked that stability holds more generally within the class of actions on metric spaces that are {\em expansive-hyperbolic}, a class of actions that includes boundary actions of fundamental groups of closed negatively curved manifolds.   This program was worked out in detail only quite recently by Kapovich-Kim-Lee \cite{KKL}, who show that what they call {\em S-hyperbolic} actions (a weakening of Sullivan's expansivity-hyperbolicity condition) are stable under perturbation with respect to a Lipschitz topology.   While $S$-hyperbolic actions represent a broader class than those studied here, in the specific case of boundary actions of fundamental groups of negatively curved manifolds our result is stronger and quite different in spirit.  We do not aim to preserve ``hyperbolic''-like behaviour, and consider perturbations in the $C^0$-topology, which can be much more violent and introduce wandering domains.   Thus, one can view Theorem \ref{thm:local_rig} as a strict strengthening of Kapovich-Kim-Lee's topological stability for the restricted case of boundary actions.

\paragraph{Sharpness}  As hinted above, one cannot replace ``factor of" with ``conjugate to" in Theorem \ref{thm:local_rig}. In Section \ref{sec:blowup}, we show that nearby, non-conjugate topological factors do occur for boundary actions of closed negatively curved manifolds.  We give two sample constructions.  One comes from {\em Cannon--Thurston} maps, special to the case where $M$ is a hyperbolic $3$-manifold, and the other is a general ``blow-up" type construction, applicable to $C^1$ examples in all dimensions.

\paragraph{2. Global rigidity of slithering actions} 
In the case where $\dim(M) = 2$, and hence $\partial_\infty(M) = S^1$, a stronger global rigidity result for boundary actions of surface groups was proved by the second author in \cite{Mann} (see also \cite{BowdenGT},  \cite{Matsumoto16}).  Using the techniques of Theorem \ref{thm:local_rig} we can recover this, and in fact generalize it to the broader context of group actions on $S^1$ arising from {\em slitherings}  associated to skew-Anosov flows on 3-manifolds, in the sense of Thurston \cite{ThurstonSlithering}.   As we discuss in the next paragraph, these flows are basic examples in hyperbolic dynamics.  Our rigidity result is the following.   

\begin{theorem}[Global rigidity of skew-Anosov slithering actions]\label{thm:slith_rigid}
Let $\F^s$ be the weak stable foliation of a skew-Anosov flow on a closed $3$-manifold $M$, and $\rho_s: \pi_1M \rightarrow \Homeo_+(S^1)$ the associated slithering action.  Then the connected component of $\rho_s$ in $\Hom(\pi_1M, \Homeo_+(S^1))$ consists of representations ``semi-conjugate" to $\rho_s$ in the sense of Ghys \cite{Ghys87}. 
\end{theorem}

\noindent 
Definitions and properties of skew-Anosov flows and slitherings are recalled in Section \ref{sec:skew}.  Note that the notion of semi-conjugacy of circle maps in the statement above is {\em not} the same as in the definition of topological factor; unfortunately the terminology ``semi-conjugacy" in this sense has also become somewhat standard.  To avoid confusion, we will follow \cite{MannWolff} and use the term {\em weak conjugacy} for this property of actions on the circle.  It has also been referred to as ``monotone equivalence'' by Calegari.  

A consequence of the above theorem is a new, independent proof of the main result of \cite{Mann} on global $C^0$ rigidity of geometric surface group actions on $S^1$.  See Corollary \ref{cor:Mann} below.

\paragraph{3. Inequivalent flows on a common manifold} 
Anosov (or {\em uniformly hyperbolic}) flows are important examples of dynamical systems, due to their stability: as originally shown by Anosov, $C^1$-small perturbations of these flows give topologically conjugate systems. Classical examples in dimension 3 include suspension flows of hyperbolic automorphisms of tori, and geodesic flows on the unit tangent bundles of hyperbolic surfaces.    The general problem of which manifolds admit Anosov flows, and the classification of such flows, is a fundamental problem in both topology and dynamics. 

The first exotic examples of Anosov flows were given by Franks and Williams \cite{FW}.  They produced non-transitive examples of flows that have separating transverse tori. Handel and Thurston \cite{HT} then gave new transitive examples, and their work planted the seeds for the definition of a general procedure (namely, the Fried--Goodman Dehn surgery) to produce new flows from old ones, later used to 
give the first examples of Anosov flows on hyperbolic 3-manifolds.

After existence, the next natural question regarding Anosov flows on a given manifold is that of {\em abundance}: how many Anosov flows, up to topological equivalence, does a given manifold support?  Results of Ghys \cite{Ghys84} and recently announced work of Barbot-Fenley \cite{BaFe} imply that principal Seifert fibered spaces admit 
at most two distinct Anosov flows up to equivalence, (cf.\ Remark \ref{rem:equiv_lifts} below). However, the case of graph manifolds, or more generally manifolds with non-trivial JSJ-decompositions, is less rigid and there are indeed examples that exhibit abundance.
The first example of a closed 3-manifold admitting at least two distinct Anosov flows was given by Barbot \cite{BarbotCAG}, and  Beguin--Bonnati--Yu \cite{BeguinBonattiYu} found examples of manifolds admitting $N$ distinct Anosov flows for arbitrarily large $N$. All these examples occur on manifolds with non-trivial JSJ-decompositions and have many (incompressible) transverse tori. This leaves open the question of the abundance for hyperbolic manifolds, which appears as Problem 3.53~(C), attributed to Christy, in Kirby's problem list \cite{Kirby}

Using techniques developed for the proof of Theorem \ref{thm:slith_rigid}, we prove the following existence result for flows, thus resolving this problem.  
\begin{theorem}[Christy's problem] \label{thm:InequivalentAnosov}
For any $N \in \N$, there exist closed, hyperbolic $3$-manifolds that support $N$ Anosov flows that are distinct up to topological equivalence.
\end{theorem}
\noindent The hyperbolic manifolds in Theorem \ref{thm:InequivalentAnosov} and the flows are described explicitly, using Dehn filling constructions.  
In addition to residing on hyperbolic manifolds, our examples contrast with those of Beguin--Bonnati--Yu in that they are {\em skew}.  They also have the further property of being {\em contact Anosov}, meaning that they are Reeb Flows for certain contact structures on these manifolds and are in particular volume preserving.  (See \cite{Foulon_Hasselblatt} for a general contstruction of contact Anosov flows by Dehn surgery.)

\paragraph{4. Topological stability of geodesic flows}
A major tool in the proof of Theorem \ref{thm:local_rig} is a ``straightening" result for quasi-geodesic flows.  This technique can also be used to recover the topological stability result for Anosov flows of Kato--Morimoto \cite[Theorem A]{KatoMorimoto} in the special case of the geodesic flow on a compact manifold of negative curvature.   

\begin{theorem}[Alternative proof of \cite{KatoMorimoto}, special case]  \label{cor:straightening}
Let $M$ be a manifold of negative curvature and $\Phi_t$ the geodesic flow on $\UT M$.  There exist $\epsilon, R > 0$ such that, if $\Psi_t$ is a flow such that each flowlines of $\Psi_t(x)$ remains $\epsilon$-close to the flowline $\Phi_t(x)$ for $t \in [0,R]$, then there is a continuous function $p(x, t)$ on $\UT M \times \R$ and surjective map $h: \UT M \to \UT M$ such that 
$h \circ \Psi_t(x) = \Phi_{p(x,t)} \circ h(x)$.    
\end{theorem} 

We note also that a related, and more direct, notion of quasi-geodesic straightening appears also in Ghys' work \cite{Ghys84} on Anosov geodesic flows.  However, despite this parallel the proofs are essentially different.  
 
\paragraph{Outline}
\begin{itemize}[leftmargin=*]
\item  Section \ref{sec:prelim} covers general background on foliations, suspensions, and large-scale geometry in negative curvature.  
\item Section \ref{sec:boundary_rig} is devoted to the proof of Theorem \ref{thm:local_rig}, followed by Theorem \ref{cor:straightening} 
\item In Section \ref{sec:blowup} we construct examples of non-conjugate actions that are $C^0$-close to the boundary actions, illustrating some of the pathological behaviour that can occur despite topological stability. 
\item In Section \ref{sec:skew} we recall the necessary background on skew-Anosov flows and slitherings,  prove Theorem \ref{thm:slith_rigid} and derive global rigidity for lifts under finite covers of the boundary action of a surface group.  
\item Section  \ref{sec:InequivalentAnosov} constructs 3-manifolds with inequivalent skew-Anosov flows, proving Theorem \ref{thm:InequivalentAnosov}.
\end{itemize}

\paragraph{Acknowledgements}
Parts of this work were completed during the Matrix centre workshop ``Dynamics, foliations, and geometry in dimension 3."
K.M. was partially supported by NSF grant DMS-1606254, CAREER award DMS-1844516, and a Sloan fellowship.  We thank B. Tshishiku for explaining some of the machinery in  the proof of  Proposition \ref{prop:blowup} and J. Manning for the reference to Hass and Scott's work on disk flow.   We are indebted to T. Barthelm\'{e} for helpful discussion, and C. Bonatti and F B\'eguin, and B. Yu for comments on an early draft of the manuscript. We also thank T. Barbot and S. Fenley for communicating their recent work as well as the anonymous referee whose detailed comments greatly improved the exposition.
\section{Preliminaries} \label{sec:prelim}
This section contains some general background material on the setting of our work, the results and framework summarized here will be used throughout.  

\subsection{Suspension foliations, flat bundles and holonomy}  \label{sec:suspension}

For an oriented manifold $N$, we let  $\Homeo_+(N)$ denote the group of orientation-preserving homeomorphisms of $N$.  For a group $\Gamma$,  the {\em representation space} $\Hom(\Gamma, \Homeo_+(N))$ is the space of homomorphisms $\Gamma \to \Homeo_+(N)$ equipped with the compact-open topology.   The case of particular interest to us, from a foliations perspective,  is when $\Gamma = \pi_1(B)$ is the fundamental group of a closed manifold $B$.  In this case one may form the {\em suspension} of a representation $\rho \in \Hom(\Gamma, \Homeo_+(N))$.  The suspension is a foliated $N$-bundle over $B$ with total space given by the quotient 
\[ E_\rho := (\tilde{B}  \times  N)/\pi_1(B), \]
where $\pi_1(B)$ acts diagonally by $\rho$ on $N$ and by deck transformations on the universal cover $\tilde{B}$ of $B$.   {\em Horizontal leaves} are  subsets of the form $\tilde{B} \times \{p\} \subset N \times \tilde{B}$.  The diagonal $\pi_1(B)$-action maps horizontal leaves to horizontal leaves, so the foliation of $\tilde{B}\times N$  by horizontals descends to a foliation on $E_\rho$ transverse to the fibers of the bundle $E_\rho \to B$.  
In our intended applications, foliations are always co-oriented and representations have image in the group of orientation-preserving homeomorphisms of $N$.  Though this is not strictly necessary for much of the background discussed here, we will take it as a standing assumption from here on.

We will typically use the notation $E_\rho$ to denote this foliated suspension space, and use other notation (e.g. simply $M$) when we wish to forget the transverse foliation on it. 

\subsection{Manifolds of negative curvature and boundaries at infinity}  \label{sec:neg_manifolds}

We briefly summarize standard results on manifolds of negative curvature that will be used in the sequel.   Further background can be found in standard references such as \cite{BGS, BH}.

Let $M$ be a closed Riemannian manifold of negative curvature.  Then its universal cover $\tilde{M}$ is a Hadamard manifold of pinched negative curvature.  In particular, it is uniquely geodesic and is a $\delta$-hyperbolic space for some $\delta$.  
Any $\delta$-hyperbolic space has a compactification by a ``boundary at infinity".  In the case of interest to us, this boundary is topologically a sphere of dimension $\dim(M) -1$.  Points on the boundary correspond to equivalence classes of geodesic rays, where two unit speed geodesics $c_1$ and $c_2: [0,\infty) \to \tilde{M}$ are equivalent if the distance $d(c_1(t), c_2(t))$ is uniformly bounded.   See \cite[III.H.3]{BH} for a general introduction in the $\delta$-hyperbolic setting, or \cite{BGS} for the Hadamard manifold setting. 
One way to specify the topology on $\partial_\infty \tilde{M}$ is as follows: fixing $x \in \tilde{M}$ and a geodesic ray $\alpha$ from $x$ to a point $\xi \in \partial_\infty \tilde{M}$, define neighborhoods $U_{r, d}(\alpha)$ of $\alpha$ to be the set of (the equivalence classes of) geodesic rays based at $x$ that stay distance at most $d$ from $\alpha$ on a ball of radius $r$ about $x$.   Such sets form a basis for the topology.  
Equivalently, one may take an exhaustion of $\tilde{M}$ by compact sets $K_i$, fix any $d>0$, and define neighborhoods $U_i(\alpha)$ of a geodesic $\alpha$ to be the sets of geodesic rays that stay within distance $d$ of $\alpha$ on $K_i$.  These also form a basis for the topology.  
Deck transformations of $\tilde{M}$ act by isometries, sending geodesics to geodesics, and this extends to an action by homeomorphisms on the boundary.  

Given a unit-speed geodesic ray $\alpha: [0, \infty) \to \tilde{M}$, the {\em Busemann function} $B_\alpha : \tilde{M} \to \R$  is defined by
\[B_\alpha (x) = \lim \limits_{t \to \infty} \left(d(\alpha(t), x) - t\right).\]   
Level sets of $B_\alpha$ are called {\em horospheres}.    Busemann functions on smooth Hadamard manifolds are always $C^2$ (as was proved in \cite{HIH}) and the horospheres $B_\alpha$ are perpendicular to geodesics.  In our setting of pinched negative curvature and bounded geometry -- this comes from the fact that the metric on $\tilde{M}$ is lifted from the compact manifold $M$ -- one can show the Busemann functions $B_\alpha$ are in fact smooth, although we will not need to use this higher regularity.      

In the case where $M$ is a surface, the boundary at infinity can be used to give a convenient description of the unit tangent bundle of $\tilde{M}$.  To any distinct triple of points $(\xi, \eta, \nu)$ in $(\partial_\infty \tilde{M})^3$, one can associate the tangent vector to the (directed) geodesic from $\xi$ to $\eta$ at the unique point $p$ such that the geodesic from $p$ to $\nu$ is orthogonal to the geodesic with endpoints $\xi$ and $\eta$.   This assignment defines a homeomorphism between the space of distinct triples in $\partial_\infty \tilde{M}$ and $\UT\tilde{M}$.   
In general, even for higher dimensional compact manifolds, the action of $\pi_1M$ on the space of distinct triples of $\partial_\infty \tilde M$ is properly discontinuous and cocompact. The reader may consult \cite[Prop 1.13]{Bowditch} for a proof phrased there in terms of the action on the Gromov boundary of a hyperbolic group.
More generally, a group acting on a space such that the induced action on the space of distinct triples is properly discontinuous and cocompact is said to be a {\em uniform convergence group}.  The following well-known property applies to any uniform convergence group action, but we state it in the form which will be useful to us later on.    

\begin{proposition}[\cite{Bowditch} Prop 3.3] \label{prop:convergence_group}
For each $x \in \partial_\infty \tilde{M}$, there exists distinct $p, q \in \partial_\infty \tilde{M}$ and a sequence of elements $\gamma_n \in \pi_1M$ such that $\gamma_n(x) \to p$ and $\gamma_n(y) \to q$ for all $y \neq x$.  
\end{proposition} 

Points $x \in \partial_\infty \tilde{M}$ with the property above are called {\em conical limit points} of the action.   A proof and further discussion can be found in \cite[\S 3]{Bowditch}.

\subsection{Geodesic flow} \label{sec:flow}  
Associated to the geodesic flow on the unit tangent bundle $\UTtM$ on the universal cover of a manifold of negative curvature are two transverse foliations, each of codimension $\dim(M) -1$.   
The leaf space of each can be identified with $\partial_\infty \tilde{M}$.  The {\em weak stable} foliation, denoted $\F^s$, has leaves $L^s(\zeta)$, for $\zeta \in \partial_\infty \tilde{M}$, consisting of the union of all geodesics with common forward endpoint $\zeta$.  The {\em weak unstable} foliation $\F^u$ consists of leaves $L^u(\xi)$ formed by geodesics with common negative endpoint $\xi$.   Both descend to foliations on $\UT M$.  

``Stable" and ``unstable" here have a precise dynamical meaning  -- the geodesic flow in negative curvature is {\em Anosov}, and the weak stable (resp.\ weak unstable) leaves consist of geodesics that converge (resp.\ diverge) exponentially (see the proof of Lemma \ref{lem:endpoint_cont} below), but we do not need any further dynamical framework at the moment, and defer a more detailed discussion to  Section \ref{sec:skew}.   
What we will use is that $\F^s$ and $\F^u$ are transverse, and also that these foliations can be described naturally in terms of the suspension of the boundary action of $\pi_1M$, as we explain now.   

The tangent bundle $\UTtM$ may also be canonically identified with $\tilde{M} \times \partial_\infty \tilde{M} = \tilde{M} \times S^{n-1}$ via the {\em positive endpoint map} which assigns to each unit tangent vector $v$ the forward endpoint of the oriented geodesic tangent to $v$.  Under this identification the horizontal sets $\tilde{M} \times \{p\}$ are the leaves of $\F^s$ and the natural projection $\tilde{M} \times \partial_\infty \tilde{M} \to \UT M$ is the quotient via the diagonal action of $\pi_1M$ by deck transformations on $\tilde{M}$ and the boundary action on $S^{n-1}$.  Thus, the suspension foliation of the boundary action gives the weak stable foliation of  geodesic flow, or equivalently {\em the holonomy of $\F^s$ is the boundary action}. 
If instead one uses the negative endpoint of oriented geodesics to identify $\UTtM$ with $\tilde{M} \times \partial_\infty \tilde{M}$, the suspension of the boundary action is the weak unstable foliation.   

\subsection{Quasi-geodesics.}

Let $c \geq 0, k\geq 1$.  A curve $\alpha$ in a metric space $X$ is a {\em (c,k) quasi-geodesic} if 
\[ \tfrac{1}{k} d(\alpha(x), \alpha(y)) - c \leq  |x-y| \leq  k \,d(\alpha(x), \alpha(y)) + c\] holds for all $x, y$ in the domain of $\alpha$.  
Often we will work with unparametrized rectifiable curves in $X$. Such a curve is quasi-geodesic if its arc length parametrization is.  
We recall two well-known and useful properties of quasi-geodesics. 

\begin{lemma}[Local-to-global principle, see \cite{CDP} Theorem 1.4] \label{lem:local_to_global}
Let $X$ be a $\delta$-hyperbolic metric space.  For any $c \geq 0, k \geq 1$, there exists $L>0$ and $c', k'$ such that every curve which is a $(c,k)$ quasi-geodesic on each subsegment of length $L$ is globally a $(c', k')$ quasi-geodesic.
\end{lemma}

\begin{lemma}[Quasi-geodesics are close to geodesics, see \cite{BH} III.H.1.7] \label{lem:close_to_geod}
Let $X$ be a $\delta$-hyperbolic space.  
There exists a constant $R = R(\delta, c, k)$ such that if $\alpha$ is a $(c, k)$ quasi-geodesic segment in $X$, then the image of $\alpha$ lies in the $R$-neighborhood of the geodesic segment joining its endpoints.
\end{lemma} 

\noindent It follows from this latter point that, provided a metric space $X$ is $\delta$-hyperbolic, each (oriented) bi-infinite quasi-geodesic $\alpha$ in $X$ has a unique bi-infinite geodesic at bounded distance.  The {\em positive} and {\em negative} endpoints of $\alpha$ are defined to be the 
positive and negative endpoints of this geodesic, denoted $e^+(\alpha)$ and $e^-(\alpha)$, respectively. 
Since quasi-isometries send quasi-geodesics to quasi-geodesics, this means that continuous quasi-isometries of $X$ extend to continuous maps on $\partial_\infty X$.  In particular, when $X = \tilde{M}$, not only deck transformations, but all lifts of homeomorphisms of $M$ to $\tilde{M}$ induce homeomorphisms of $\partial_\infty \tilde{M}$.   

A {\em (unparametrized) quasi-geodesic flow} of a metric space $X$ is a 1-dimensional foliation whose leaves are quasi-geodesics.  The flow is {\em uniform} if there exist $k\geq1, c\geq0$ such that each leaf is a $(c,k)$ quasi-geodesic.  If $\Gamma$ is a group that acts properly discontinuously and cocompactly on a space $X$ and $\F^{QG}$ is a quasi-geodesic foliation such that the action of $\Gamma$ sends leaves to leaves, then local-to-global principle  implies that $\F^{QG}$ is automatically uniform.  

Using Lemma \ref{lem:close_to_geod}, and the definition of the topology on $\partial_\infty X$ described above, one easily attains the following.  

\begin{lemma} \label{lem:fellow_travel}
Let $X$ be $\delta$-hyperbolic and let $\alpha$ be a $(k,c)$ quasi-geodesic ray based at $x_0$.  Given a neighborhood $U$ of $e^+(\alpha) \in \partial_\infty X$, and constant $d>0$, there exists a compact set $K$ such that, if $\beta$ is any $(k,c)$ quasi-geodesic ray that is distance at most $d$ from $\alpha$ on $K$, then $e^+(\beta) \in U$.    
\end{lemma}

The same evidently holds for $e^-$.  From this, one may derive the fact that endpoint maps are {\em continuous} on the space of $(k,c)$ quasi-geodesics equipped with the compact-open topology, and hence descend to continuous maps on the leaf space of a uniform quasi-geodesic foliation.   The following alternative proof of this fact appears essentially in \cite{Calegari}; we include it as it gives another helpful illustration of the behavior of uniform quasi-geodesics in negative curvature.  

\begin{lemma} \label{lem:endpoint_cont}
Let $\F$ be an oriented, uniform quasi-geodesic foliation of the universal cover of a compact manifold of negative curvature.  Then the endpoint maps, considered as functions on the leaf space of $\F$, are continuous.  
\end{lemma}

\begin{proof}
Suppose that $\ell_n$ is a sequence of leaves of $\F$ that converge uniformly on compact sets to a leaf $\ell_\infty$.  Following the discussion after Lemma \ref{lem:close_to_geod}, there exists $D>0$ (depending on the curvature of $M$ and the quasi-geodesic constants of leaves) such that each $\ell_n$ lies in the $D$-neighborhood of a unique geodesic $\gamma_n$.   It follows that the $\gamma_n$ coarsely converge on compact sets: after passing to a subsequence, we may assume that there is a length $n$ segment of $\gamma_n$ which lies in the $3D$-neighborhood of $\gamma_\infty$.   Since geodesics in negative curvature have exponential divergence\footnote{Recall that a {\em divergence function} for a metric space $X$ is a function $\Delta: \N \to \R$ such that, for any geodesics $c_1, c_2: [0, t] \to X$ with $c_1(0) = c_2(0)$, and any $r, R \in \N$; if $R + r  \leq t$ and $d_X(c_1(R), c_2(R)) > e(0)$ then any path from $c_1(R + r)$ to $c_2(R+r)$ outside the ball $B(c_1(0), R + r)$ must have length at least $\Delta(r)$.    Any $\delta$-hyperbolic space has an exponential divergence function. (See \cite[III.H.1.25]{BH} for a proof).}
 this implies that $\gamma_n$ lies in a $3D e^{-\lambda n}$ neighborhood of $\gamma_\infty$ on a segment of length $n/2$, for some $\lambda>0$.    Thus, the $\gamma_n$ converge and so $e_+(\lambda_n) = e_+(\gamma_n)$ converges to $e_+(\gamma_\infty)$.  
\end{proof}

This concludes the preliminary material required for the proof of Theorem \ref{thm:local_rig}.  Further material on Anosov flows and an introduction to ``slitherings" will be given in Section \ref{sec:skew} where it is needed.

\section{Proof of Theorem \ref{thm:local_rig}}  \label{sec:boundary_rig}

\subsection{Construction of a well-behaved leafwise immersion} \label{sec:leafwise_imm}

The first step in the proof of Theorem \ref{thm:local_rig} is to construct a well behaved map from the suspension bundle $E_{\rho}$ (see Section \ref{sec:suspension}) of a perturbation $\rho$ of the boundary action, to the unit tangent bundle of $M$.   This map will send fibers to fibers, and send leaves of the horizontal foliation on the suspension $E_{\rho}$ to $C^1$ submanifolds such that the tangent distribution of each leaf is $C^0$ close to the distribution given by the weak-stable distribution of the geodesic flow on $M$.

\begin{lemma}   \label{lem:leafwise_imm} 
Let $M$ be a compact negatively curved manifold, and let $\rho_0$ be the boundary action.  There exists a neighborhood $U$ of $\rho_0$ in $\Hom(\pi_1M, \Homeo_+(S^{n-1}))$  and a continuous assignment $\rho \mapsto f_{\rho}$ from $U$ to the space of continuous maps $\tilde{M} \times S^{n-1} \to \UTtM$ with the following properties:

\begin{enumerate}
\item $f_{\rho_0}$ is the identity map,
\item for all $\rho \in U$, the map $f_{\rho}$ covers the identity $\tilde{M} \to \tilde{M}$ mapping fibers to fibers (although it is not required to be injective on any fiber), 
\item the image of each horizontal leaf $\tilde{M} \times \{p\}$ under $f_{\rho}$ is a  $C^1$ submanifold of $\UTtM$,  
\item the map $f_{\rho}$ is $\pi_1 M$-equivariant, so descends to a map $E_{\rho} \to \UT M$, and
\item given $\epsilon > 0$, and $R > 0$, by choosing $U$ sufficiently small we can ensure that the image of any leaf $\tilde{M} \times \{p\}$ under $f_{\rho}$ has tangent distribution $\epsilon$-close (in the $C^0$-sense, from the metric lifted from $M$) to the weak-stable distribution over any ball of radius $R$ in $\tilde{M}$.  Equivalently, by choosing $U$ sufficiently small, we can ensure for each $p$ that the restriction of $f_\rho$ to $B \times \{p\}$, where $B$ is any $R$-ball in $\tilde{M}$, is $\epsilon$-close in the $C^1$-topology to a weak-stable leaf in $\UTtM$.  
\end{enumerate}  
Finally, in the case where $M$ is a surface and $\partial_\infty \tilde{M} = S^1$, we may additionally take $f_\rho$ to be a homeomorphism for all $\rho \in U$.  
\end{lemma} 

\begin{remark} \label{rem:general_bundle}
The construction in the proof generalizes to other foliated bundles than $\UT(M)$.  What we use in the construction is compactness of $M$, the linear structure of the tangent bundle, the $C^{1,0+}$ regularity of the weak-stable foliation on $\UT(M)$, and the fact that $\rho$ is a $C^0$-small perturbation of the holonomy of this bundle.  
\end{remark} 

\begin{proof}[Proof of Lemma \ref{lem:leafwise_imm}]
Let $\tau$ be a smooth triangulation of $M$, and let $\tilde{\tau}$ denote its lift to $\tilde{M}$.
For the proof, we define $f_{\rho}$ first on the spheres over the vertices of $\tau$, then extend using a partition of unity to ``interpolate'' between vertices.  

\noindent \textbf{Set-up. } Let $O_v$ denote the open star of a vertex $v \in \tau$.  The set $\mathcal{O} : = \{ O_v \mid v \text{ vertex of } \tau \}$ is an open cover of $M$ whose nerve agrees with $\tau$.    Let $\tilde{\mathcal{O}}$ be the lift of $\mathcal{O}$ to $\tilde{M}$.  Take a partition of unity $\{ \sigma_v : v \text { vertex of } \tau \}$ subordinate to $\mathcal{O}$ and pull this back to a $\pi_1 M$-equivariant partition of unity $\{ \tilde{\sigma}_v : v \text { vertex of }  \tilde{\tau} \}$ on $\tilde{M}$ subordinate to $\tilde{\mathcal{O}}$.  \medskip 

\noindent \textbf{Step 1: Define $f$ over vertices of $\tilde \tau$.}
We recall first the structure available to us.  
$\UTtM$ is a sphere bundle over $\tilde{M}$ with a natural action of $\pi_1 M$ by diffeomorphisms, as well as a linear (vector bundle) structure coming from the Riemannian metric on $M$.  
The lift of the weak-stable foliation $\tilde{\F}^s$ is transverse to the $S^{n-1}$ fibers of $\UTtM$ and invariant under $\pi_1 M$, so gives $\UT M$ the structure of a foliated bundle, and the holonomy of this foliation is the boundary action $\rho$.

Let $D \subseteq \widetilde{M}$ be a connected fundamental domain for the $\pi_1 M$-action.   Fix a basepoint $x \in D$ that is a vertex in $\tilde \tau$ and let $S_x$ denote the fiber over $x$ in $\UTtM$.  For a perturbation $\rho$ of $\rho_0$, we may build a foliated bundle $E_{\rho}$ with holonomy $\rho$ by taking the quotient of $\tilde{M} \times S_x \cong \tilde{M} \times S^{n-1}$ by the diagonal action of $\pi_1 M$ by deck transformations of $\tilde{M}$ and the action $\rho$ on $S_x \cong S^{n-1}$.  
Define $f_{\rho}$ on $ \{x\} \times S_x$ to agree with the identity map $S_x \to S_x$.  

Now we define $f_\rho$ over other vertices.  
For a vertex $v$ of $\tilde\tau$, and $S_v$ the fiber over $v$ in $\UTtM$, let $\pi_v: S_v \to S_x$ denote the homeomorphism obtained by sending a point $y \in S_v$ to the unique point of $S_x$ lying in the same leaf of $\tilde{F}^s$ as $y$.   
Now for each vertex $v$ of $\tilde\tau$ in $D$, define $f_{\rho}$ on $\{v \} \times S^{n-1}$ by $f_{\rho}(v, p) = S_v^{-1}f_{\rho}(x, p)$.  
In other words, we define $f_{\rho}$ on the fibers over vertices in $D$ so that it takes points on the same leaf of the horizontal foliation on $\tilde{M} \times S^{n-1}$ to points on the same leaf of $\tilde{F}^s$.  

There is a unique $\pi_1 M$-equivariant extension of this map to one defined on the union of fibers over all vertices of $\tilde{\tau}$.  
Concretely this is given as follows:  For a vertex $v \in D$, an element $g \in \pi_1 M$, and a point $(gv, p)$ in $\{gv\} \times S^{n-1}$, define 
\[ f_{\rho}(gv, p) =  \left(gv, \pi_{gv}^{-1} \rho_0(g)\rho(g)^{-1} (p) \right) \]
i.e., it is the unique point on the fiber $S_{gv}$ contained in the same leaf of $\tilde{F}^s$ as the point $\rho_0(g)\rho(g)^{-1}(p) \in S_x$.   

\begin{remark*}
Note that, given any finite set of elements of $g$, and any $\epsilon >0$, we can choose a neighborhood $U$ of $\rho_0$ sufficiently small so that $\rho_0(g)\rho(g)^{-1}(p)$ lies within distance $\epsilon$ of $p$.  
\end{remark*}

\noindent \textbf{Step 2: Extend over $\tilde{M} \times S^{n-1}$.}
For a vertex $v$ of $\tilde{\tau}$, let $O_v$ denote its open star, and let $F_v: O_v \times S^{n-1} \to \UTtM$ be the map that covers the identity on $O_v \subset \tilde{M}$, agrees with $f_{\rho}$ on $\{v\} \times S^{n-1}$, and sends horizontal leaves $O_v \times \{p\}$ to leaves of $\F^s$.  

Using the linear structure on $\UTtM$, we now extend $f_{\rho}$ to be defined everywhere by using the partition of unity to interpolate between the maps $F_v$.  In detail, on a horizontal leaf $\tilde{M} \times \{p\} \subset \tilde{M} \times S^{n-1}$ define for each point $(z, p)$ a (not necessarily unit) tangent vector to $\tilde M$ at $z$ by
\[ V(z,p) := \sum_{v \text { vertex of } \tilde\tau} \widetilde{\sigma}_v(z) F_v(z, p). \]
Note that this is a finite sum involving only vertices of $\tau$ whose star contains $z$.  The remark above together with the fact that $F_v$ is $\pi_1 M$-equivariant ensures that if $\rho$ is chosen sufficiently close to $\rho_0$, then $F_v(z, p)$ is never zero.  
Thus, we may define 
\[ f_{\rho}(z, p) : = \frac{V(z,p)}{|| V(z,p) ||} \in \UTtM. \]
Since $\tilde{\sigma}$ is smooth and $\mathcal{F}^s$ is a $C^{1,0+}$ foliation, for fixed $p$ the maps $V(z,p)$, and hence also the maps $f_{\rho}(z, p)$ are of class $C^1$. Moreover, these maps vary continuously in $z$ in the $C^1$ topology and as all functions in the definition are $\pi_1M$-equivariant, so is $f_{\rho}$.  

Continuity in $\rho$ also follows readily from the definition. In fact, this map is continuous at the point $\rho_0$ with respect to leafwise uniform convergence in the $C^1$ topology on compact sets of uniform size.  For if $\rho_n \to \rho_0$, then the fact that $\tilde{F}^s$ is of class $C^{1,0+}$ implies that all the summands in the definition of $V(z,p)$ will $C^1$-converge uniformly on any fixed compact set.   By $\pi_1 M$--equivariance we then get uniform converge on balls of a fixed radius about any point. 

Finally, in the case where $\dim(M) = 2$, and $\rho$ is an action on the circle, one can avoid the normalisation of $V(z,p)$ in the definition of $f_\rho$ and make a separate argument to produce a map which descends to a {\em homeomorphism} $E_\rho \to \UT M$, as follows.  Consider the universal cover $\widetilde{\UT M}$, which is an $\R$-bundle over $\tilde{M}$ and a fiberwise cover of $\UTtM$.   A perturbation $\rho$ of $\rho_0$ gives rise to a unique perturbation of the holonomy of $\widetilde{\UT M}$ in the group $\Homeo_\Z(\R)$ of homeomorphisms of the fiber $\R$ that commute with the covering map $\R \to S^1 \cong \R/\Z$. Repeating the same construction as above to define $F_v$ using the lifted actions, we can then average the maps $F_v$ using the partition of unity $\tilde{\sigma}_v$ and the natural Lie group structure on the fiber $\R$ lifted from $S^1$.   The resulting map $\tilde{M} \times \R \to \widetilde{\UT M}$ will additionally commute with the fiberwise covering map, and descend to a map $f_\rho: \tilde{M} \times S^1 \to \UTtM$ with the required properties (1)-(5).  
The fact that points on $\R$ are totally ordered, and that orientation-preserving homeomorphisms and our averaging trick are all order-preserving means that, with this construction, $f_{\rho}$ will be bijective.  It is easy to verify that its inverse is also continuous, hence $f_{\rho}$ is $\pi_1 M$ equivariant homeomorphism.  
\end{proof}

\begin{figure}
  \labellist 
  \hair 2pt
   \pinlabel $S^{n-1}$ at 25 260
     \pinlabel $L$ at 215 160
 \pinlabel $\ell$ at 100 145
    \pinlabel $\tilde M$ at 215 50
       \pinlabel $f_\rho$ at 280 130
       \pinlabel $\longrightarrow$ at 280 110
         \pinlabel $f_\rho(L)$ at 575 160
    \pinlabel $\UTtM$ at 440 15
       \pinlabel $L^u(\xi)$ at 470 240
   \endlabellist
   \centerline{ \mbox{
 \includegraphics[width = 4in]{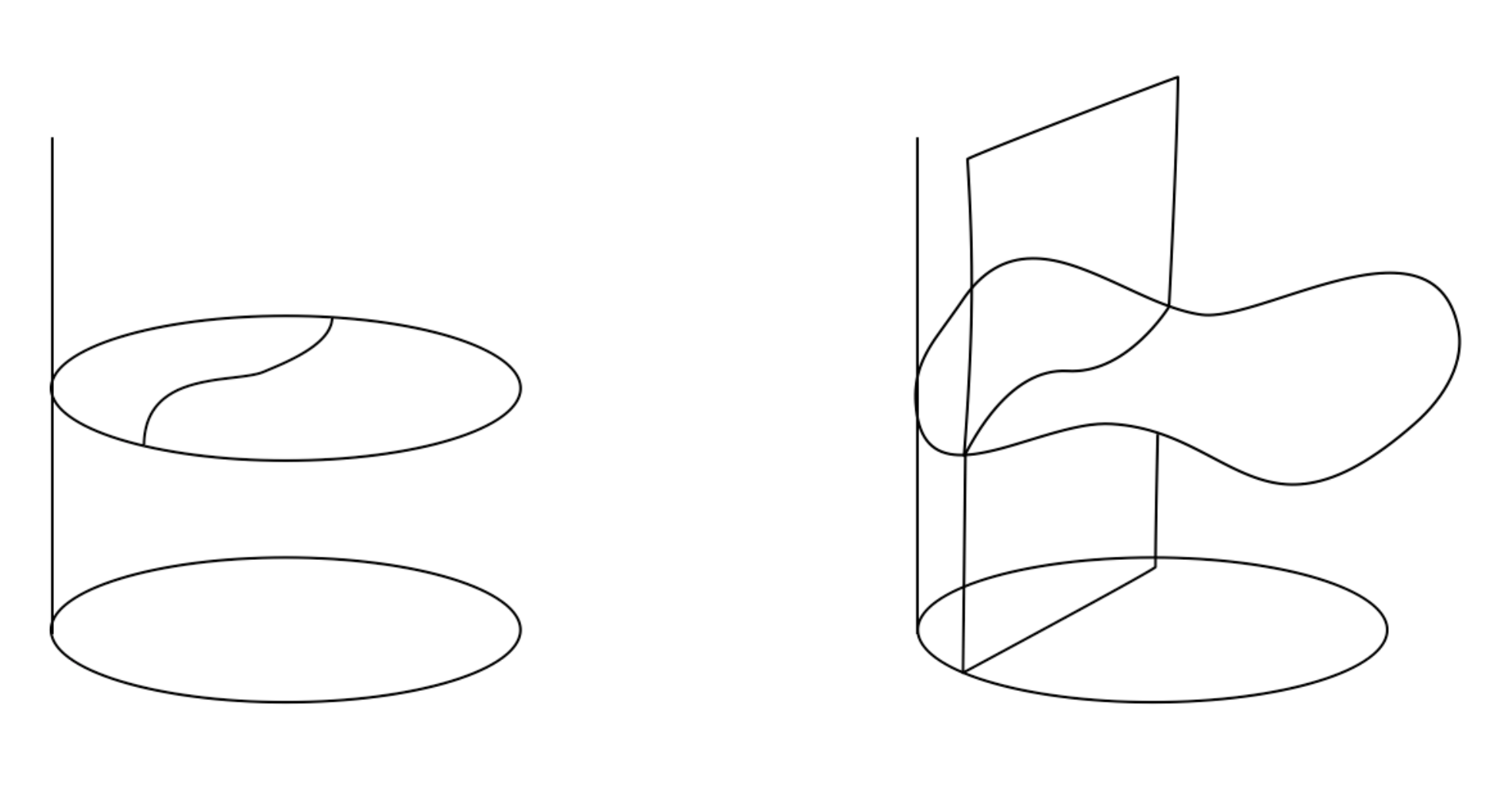}}}
 \caption{The map $f_\rho$ and a leaf $\ell$ of $\F^{QG}_\rho$}
  \label{f_rho}
\end{figure}
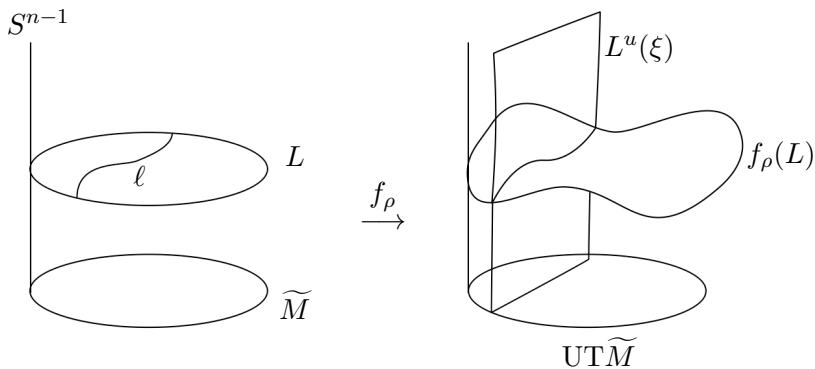

\subsection{Quasi-geodesics and endpoint maps} \label{sec:endpoint_maps}
Keeping the notation from Section \ref{sec:leafwise_imm}, $M$ denotes a compact, negatively curved Riemannian $n$-manifold, $\UTtM$ the unit tangent bundle of $\tilde{M}$, and $\rho_0$ denotes the standard boundary action of $\pi_1M$ on $S^{n-1}$.   

Let $f_\rho: \tilde{M} \times S^{n-1} \to \UTtM$ denote the $\pi_1M$-equivariant map obtained by applying Lemma \ref{lem:leafwise_imm} to a representation $\rho$ close to $\rho_0$ in $\Hom(\pi_1M, \Homeo_+(S^{n-1}))$. 
Our next goal is to use this data to produce a quasi-geodesic foliation on $\tilde{M} \times S^{n-1}$ that is $\rho$-equivariant, so descends to a foliation on the suspension $E_\rho$.  
When speaking of endpoints of geodesics, we use the fact that the canonical projection $\pi: \UTtM \to \tilde{M}$ is a quasi-isometry, giving us an identification of the Gromov boundary of $\UTtM$ with $\partial_\infty \tilde{M}$.  We use the notation from Section \ref{sec:flow}, for example $L^u(\xi)$ denotes the unstable leaf consisting of geodesics with negative endpoint $\xi$.  

\begin{lemma} \label{lem:qg_line}
If $\rho$ is sufficiently close to $\rho_0$, then for any horizontal leaf $L = \tilde{M} \times \{p\}$ of $\tilde{M} \times S^{n-1}$, and any unstable leaf 
$L^u(\xi)$ in $\UTtM$, 
the intersection $f_\rho(L) \cap L^u(\xi)$  is either empty or a connected, quasi-geodesically embedded bi-infinite line in $\UTtM$ with one endpoint equal to $\xi$. 
\end{lemma} 

Note that the intersection may indeed be empty, for instance, when $\rho = \rho_0$, the leaf $L^u(\xi)$ has empty intersection with the stable leaf $L^s(\xi)$ comprised of geodesics with forward endpoint $\xi$.  

\begin{proof}[Proof of Lemma \ref{lem:qg_line}]
Let $U$ be a small neighborhood of $\rho_0$.  
Consider the horosphere foliation $\mathcal{B}$ of $\UTtM$ whose leaves are level sets of the Busemann function $b$ of a geodesic with forward endpoint $\xi$ (see Section \ref{sec:neg_manifolds}).  
 For any leaf $L^s(\zeta)$ of the stable foliation of geodesic flow on $\UTtM$ intersecting $L^u(\xi)$, the geodesic leaf $\ell = L^u(\xi) \cap L^s(\zeta)$ is perpendicular to leaves of $\mathcal{B}$.  Since, for any leaf $L = \tilde{M} \times \{p\}$, the tangent distribution of $f_\rho(L)$ is uniformly $C^0$ close to the stable distribution (Property (5) in Lemma \ref{lem:leafwise_imm}), it follows that $L^u(\xi)  \cap f_\rho(L)$ meets leaves of $\mathcal{B}$ at angle uniformly close to $\pi/2$.  Thus, the length of the segment of a leaf of $L^u(\xi) \cap  f_\rho(L)$ between $b^{-1}(t)$ and $b^{-1}(s)$ is at most $C|t-s|$ for some constant $C>1$, which we may take to be uniform over all leaves, and this constant will only be decreased by further shrinking the neighborhood $U$ of $\rho_0$.    This shows that connected components of leaves $\ell$ are uniform quasi-geodesics.  

To show the sets $L_\xi \cap  f_\rho(L)$ are either empty or a {\em single} quasi-geodesic line, we use the fact that $\pi_1 M$ acts cocompactly on the space of distinct triples of $\partial_\infty \tilde{M}$.   As in the previous section, let $x$ be a basepoint in $\tilde{M}$.  Using cocompactness, and the fact that quasi-geodesics fellow-travel geodesics in $\tilde{M}$, we may choose $R$ large enough so that, for any distinct triple of boundary points $(\xi_1, \xi_2, \xi_3)$, there exists $g \in \pi_1M$ so that any quasi-geodesic in $\UTtM$ with constant $C$ (as above) between any pair $g(\xi_i)$ and $g(\xi_j)$ has projection to $\tilde{M}$ which passes through the ball $B_R(x)$ of radius $R$ about $x$.    

Our construction of $f_\rho$ means that, provided that $\rho$ is close enough to $\rho_0$, the image of any leaf $\tilde{M} \times \{p\}$ under $f$ will be uniformly $C^1$ close to the stable leaf through $p$ over the ball $B_{2R}(x)$.  Thus if $U$ is chosen sufficiently small, then for any leaf $L$ the projection of $L \cap L^u(\xi)$ intersected with $B_R(x)$ will either be empty or a single, connected quasi-geodesic segment.  
Suppose now for contradiction that there exist leaves $L$ and $L_\xi$ such that $L^u(\xi) \cap f_\rho(L)$ is nonempty and not connected. Let $\alpha \neq \xi$ and $\beta \neq \xi$ be endpoints of two distinct (quasi-geodesic) connected components of $L^u(\xi) \cap f_\rho(L)$.   Using our choice of $R$, find $g$ so that the quasi-geodesics components of $g (L^u(\xi) \cap f_\rho(L))$ with pairs of endpoints ($g \alpha$, $g \xi$)  and  $(g \beta, g \xi)$, respectively, satisfy the property that their projections to $\tilde{M}$ pass through $B_R(x)$.  (We allow the possibility that $\alpha = \beta$.)  But then $L^u(g \xi) \cap f_\rho(g L)$ intersects $B_R(x)$ along two {\em distinct} quasi-geodesic segments, giving the desired contradiction.  
\end{proof} 

Thus, after endowing $f_\rho(L)$ with either its induced metric or that pulled back from the projection $\pi: \UTtM \to \tilde{M}$, the sets $L^u(\xi) \cap f_\rho(L)$ (fixing $L$ and varying $\xi$) give a quasi-geodesic foliation of $f_\rho(L)$. 
We make the following orientation convention; when $\rho = \rho_0$, this exactly recovers the oriented geodesics from $\UTtM$.  

\begin{convention}[Orientation on leaves]
We orient the lines of the form $f_\rho(L) \cap L^u(\xi)$ so that their negative endpoint is $\xi$.
\end{convention}

Since $f_\rho$ covers the identity on $M$ and is $C^1$ on leaves, its restriction to each leaf $L$ is a quasi-isometry.  Thus, we may pull back the oriented quasi-geodesic foliation on each leaf $f_\rho(L)$ via the restriction of $f_\rho$ to $L$, and obtain an oriented quasi-geodesic foliation on $L$.  Doing this on all leaves gives an oriented, quasi-geodesic foliation on $\tilde{M} \times S^{n-1}$, which we denote by $\Fqg_\rho$.  Again, $\pi_1 M$-equivariance means that the quasi-geodesic constants may be taken to be uniform.    See Figure \ref{f_rho} for an illustration.  

\subsection*{Properties of $\Fqg_\rho$}
The fact that $f_\rho$ is $\pi_1M$-equivariant and that $\F^u$ is a $\pi_1M$-invariant foliation on $\UTtM$ means that the diagonal action of $\pi_1 M$ on $\tilde M \times S^{n-1}$ via deck transformations on the first factor and $\rho$ on the second permutes the leaves of $\Fqg_\rho$. 
Furthermore, $\Fqg_\rho$ has the property that each quasi-geodesic line $\ell$ in the foliation is contained in a horizontal leaf of $\tilde{M} \times \{p\}$ of $\tilde{M} \times S^{n-1}$.   Thus, such a line $\ell$ has a positive and negative endpoint on the boundary sphere $\partial_\infty \tilde{M} = S^{n-1}$, giving positive and negative \emph{endpoint maps} 
\[e_\rho^+,  e_\rho^-: \tilde{M} \times S^{n-1} \to \partial_\infty \tilde{M}  =S^{n-1}\]
which assign to a point $x$ in a leaf $\ell$ the positive and negative endpoints $e_\rho^+(x)$ and $e_\rho^-(x)$ of $\ell$, respectively.   Note that, since $f_\rho$ covers the identity map on $\tilde{M}$, one may equally well look at the sets $f_\rho(L) \cap L^u(\xi)$ or their pullbacks under $f_\rho$ to determine their endpoints.
When $\rho = \rho_0$, the foliation $\Fqg_\rho$ is the geodesic foliation of $\UTtM$, and $e_{\rho_0}^+$ and $e_{\rho_0}^-$ are the usual positive and negative endpoint maps.

Let $e_\rho^\pm$ denote the product map $(e^+, e^-)$ to $S^{n-1} \times S^{n-1}$.  The image of this map avoids the diagonal $\Delta$.  By definition, this map factors through the projection to the leaf space $\mathcal{L}(\Fqg_\rho)$ of $\Fqg_\rho$ as summarized in the diagram below.  
\[
  \begin{tikzcd}
    \tilde{M} \times S^{n-1} \arrow{d} \arrow{r}{e_\rho^\pm} & (S^{n-1} \times S^{n-1}) - \Delta \\
    \mathcal{L}(\Fqg_\rho) \arrow[swap]{ur}{\overline{e}_\rho}
  \end{tikzcd}
\]
Additionally, since $f_\rho$ is $\pi_1M$-equivariant, a straightforward verification from the definition shows that the same is true of $e^\pm_\rho$, namely 
\begin{equation} \label{eq:equivariance}
e^\pm_\rho (\gamma\cdot x, \rho(\gamma)(y)) = \gamma \cdot e^\pm_\rho(x,y) 
\end{equation} 
holds for all $(x,y) \in \tilde{M} \times S^{n-1}$ and $\gamma \in \pi_1M$, where the action on the right hand side of the equation is by the standard action of $\pi_1M$ on unparametrized geodesics in $\UTtM$,  i.e. the diagonal action of $\rho_0$ on boundary points.  

We now prove various continuity properties.

\begin{lemma} \label{lem:cont_1}
The endpoint maps $e^\pm_\rho$ and $\overline{e}_\rho$ are continuous.
\end{lemma}

\begin{proof}
Lemma \ref{lem:endpoint_cont} implies that the restriction of $\overline{e}_\rho$ to each leaf $L$ is continuous.   
We will use a similar argument to show global continuity.  It suffices to show continuity of $e^\pm_\rho$, since $\overline{e}_\rho$ is the induced map on a quotient space.  

Suppose that $x_n \to x_\infty$ is a convergent sequence in $\tilde{M} \times S^{n-1}$.  Let $L_n$ be the horizontal leaf containing $x_n$, and $L_\infty$ the leaf containing $x_\infty$.    Let $f_n$ denote the restriction of $f_\rho$ to $L_n$, considered as a topological embedding $\tilde{M} \hookrightarrow \UTtM$.  The definition of $f_\rho$ implies that the maps $f_n$ converge uniformly on compact sets of uniform diameter, i.e. for any $r > 0,\epsilon > 0$, there exists $N$ such that for all $n>N$, the restriction of $f_n$ to the $r$-ball $B_r(x_n)$ in $L_n$ is $\epsilon$-close in the $C^1$-topology  
 to the restriction of $f_\infty$ to $B_r(x_\infty)$.   
It follows that, for any fixed leaf $L^u$ of $\F^u$,  the quasi-geodesic segment $L^u \cap f_n(B_r(x_n))$ lies in some $C(r) \epsilon$ neighborhood of $L^u \cap f_\infty(B_r(x_\infty))$, where $C: [0, \infty) \to [0, \infty)$ is a continuous, increasing function (depending only on the geometry of $\F^u$), with $C(0) = 0$.  

Let $L^u_n$ be the leaf of $\F^u$ through $f_n(x_n) = f_\rho(x_n)$.  Since $f_\rho(x_n) \to f_\rho(x_\infty)$ these leaves converge on compact sets to the leaf $L^u_\infty$ through $f_\rho(x_\infty)$.  Combined with the above, we deduce that, for $n$ sufficiently large, $L^u_n \cap f_\rho(B_r(x_n))$ lies in the $2 C(r) \epsilon$ neighborhood of $L^u \cap f_\rho(B_r(x_\infty))$.    
Since these are uniform quasi-geodescis, 
Lemma \ref{lem:fellow_travel} now gives the desired continuity.   
\end{proof}

\begin{lemma} \label{lem:cont_rho}
The map $\rho \mapsto e^\pm_\rho$, defined on a neighborhood of $\rho_0$ in $\Hom(\pi_1M, \Homeo_+(S^{n-1}))$ and with image in the space of continuous maps $\tilde{M} \times S^{n-1} \to (S^{n-1} \times S^{n-1}) - \Delta$, is continuous with respect to the compact-open topology.    
\end{lemma}

\begin{proof}
This follows from continuity of $\rho \mapsto f_\rho$ and the definition of the topology on the end space.  In detail, let $\rho'$ be some fixed representation close to $\rho_0$, close enough so that $f_{\rho'}$ and the endpoint maps are defined.  
The space of continuous maps $\tilde{M} \times S^{n-1} \to (S^{n-1} \times S^{n-1}) - \Delta$ has the standard compact-open topology, so fix $K$ compact in $\tilde{M} \times S^{n-1}$ and an open set $O$ in $(S^{n-1} \times S^{n-1}) - \Delta$ containing the image of $K$.   Continuity of $\rho \mapsto f_\rho$ means that, for any $\epsilon, R>0$ if $\rho$ is chosen close enough to $\rho'$ then quasi-geodesics through points of $K$ pulled back via $f_\rho$ will remain $\epsilon$-close to quasi-geodesics pulled back via $f_{\rho'}$ on segments of length $R$.  
Lemma \ref{lem:fellow_travel} now guarantees that for $R$ large enough, the endpoints of geodesics through points of $K$ will remain in $O$. 
\end{proof}

\begin{lemma}[$e^-_\rho$ gives local parametrization of leaves] \label{lem:local_transversal}

Any local transversal for the geodesic flow $\Fqg_{\rho_0}$ will be a local transversal for any sufficiently close representation $\rho$, in particular, the leaf space $\mathcal{L}(\Fqg_\rho)$ is locally homeomorphic to $\R^{n-1} \times \R^{n-1}$.   Associating a leaf $\ell$ to the pair $(e^-_\rho(\ell), p)$, where $p \in S^{n-1}$ is the point such that $\ell$ lies in the horizontal leaf $\tilde{M} \times \{p\}$, gives a local chart for $\mathcal{L}(\Fqg_\rho)$.  
\end{lemma} 

\begin{proof} 
Continuity of $\rho \mapsto f_\rho$ and the fact that $\Fqg_\rho$ is the pullback of the intersection of (smooth) leaves $f_\rho(L) \cap L^u$ implies that a compact local transversal for $\Fqg_{\rho_0}$ will remain transverse to $\Fqg_\rho$ when $\rho$ is sufficiently close to $\rho_0$.   
By Lemma \ref{lem:qg_line}, each leaf $L^u({\xi})$ of $\F^u$ intersects a leaf $f_\rho(L)$ in a (possibly empty) quasi-geodesic with negative endpoint $\xi$.  Thus, the negative endpoint map locally gives a parametrization of the leaves of $\Fqg$ which sit inside a fixed horizontal leaf $L$.  Continuity of $f_\rho$ and the negative endpoint map means that these paramatrizations vary continuously with the leaf $L$, giving the desired local chart.  
\end{proof}

\begin{lemma} \label{lem:surjective}
If $\rho$ is sufficiently close to $\rho_0$, then $\overline{e}_\rho$ is surjective.
\end{lemma}

\begin{proof}
Take a $(2n-2)$-dimensional disc $D$ in $\tilde{M} \times S^{n-1}$ that is a local transversal for the geodesic foliation $\Fqg_{\rho_0}$, chosen large enough so that the interior of the image $\bar{e}_{\rho_0}(D) \subset (S^{n-1} \times S^{n-1}) - \Delta$ contains a compact fundamental domain $K$ for the action of $\pi_1M$ on the space $(S^{n-1} \times S^{n-1}) - \Delta$ of unparametrised geodesics in $\tilde{M}$.  

By Lemma \ref{lem:local_transversal}, if $\rho$ is sufficiently close to $\rho_0$, then $D$ will also be a local transversal for $\Fqg_\rho$, and, by continuity of the endpoint map, $\bar{e}_\rho(D)$ will be $C^0$ close to $\bar{e}_{\rho_0}(D)$ and hence also contain $K$. Since $e^\pm_{\rho}$ is $\pi_1M$-equivariant, it follows that the image of  $\overline{e}_{\rho}$ is a set that is invariant under the action of $\pi_1M$ on $(S^{n-1} \times S^{n-1}) - \Delta$.  We have just shown that it contains a fundamental domain, so the image must be everything.  
\end{proof}

The following observation  will allow us to conclude the proof by arguing that $e^+_\rho$ defines a semi-conjugacy.  

\begin{proposition} \label{prop:positive_const}
Under the hypotheses of Lemma \ref{lem:surjective}, the restriction of $e^+_\rho$ to any horizontal leaf $L$ of $\tilde{M} \times S^{n-1}$ is constant.  
\end{proposition}

The broad idea of the proof is to use $\pi_1M$-equivariance and the uniform convergence group property of the boundary action to promote a (hypothetical) leaf where $e^+_\rho$ is non-constant to one where $e^-_\rho$ is not locally injective, which would contradict the local structure given by Lemma \ref{lem:local_transversal}.

\begin{proof}[Proof of Proposition \ref{prop:positive_const}]
Suppose for contradiction that $e^+_\rho$ is nonconstant on some leaf $L$, and let $I \subset L$ be a segment such that $e^+_\rho(I)$ is a nonconstant path with 
distinct endpoints in $S^{n-1}$.  We may even take $I$ to be transverse to $\Fqg_\rho$, if desired, and the reader may find this helpful in visualizing the proof.  Let $x$ and $y$ denote the endpoints of $e^+_\rho(I)$.     Since the image of $(e^+_\rho, e^-_\rho)$ avoids the diagonal, by shrinking $I$ if needed we may further assume that $e^+_\rho(I)$ is disjoint from $e^-_\rho(I)$, in particular $x \notin e^-_\rho(I)$.  

By the uniform convergence group property of the action of $\pi_1M$ on its boundary (Proposition \ref{prop:convergence_group}), 
there exist distinct $p, q \in S^{n-1}$ and a sequence $\gamma_n \in \pi_1M$ such that $\gamma_n(x) \to p$ and $\gamma_n(z) \to q$ for all $z \neq x$.  Thus, the image $\gamma_n  e^+_\rho(I) = e^+_\rho  \rho(\gamma_n)(I)$ will contain an arc between some points $p_n$ and $q_n$, with $p_n \to p$ and $q_n \to q$; while $\gamma_n  e^-_\rho(I) = e^-_\rho  \rho(\gamma_n)(I)$ pointwise converges to $\{q\}$. 

Consider the sequence of leaves $\rho(\gamma_n)(L)$.  Since the leaf space of the horizontal foliation (on $\tilde{M} \times S^{n-1}$) is compact, after passing to a subsequence these converge to some leaf $L_\infty$.   
 It will be convenient for us to remain in a compact set of $(S^{n-1} \times S^{n-1}) - \Delta$, so fix a small open neighborhood $N$ of $\Delta$, and let $J_n$ denote the closure of the connected component of $\gamma_n  e^+_\rho(I) - N$ containing $p_n$; this is some subinterval of $\gamma_n  e^+_\rho(I)$.  
 Let $D$ be a compact, local transversal for $\Fqg_\rho$, defined in a neighborhood of some quasi-geodesic leaf lying in $L_\infty$ so that for all $n$ sufficiently large, the projection of the segments $J_n \subset \rho(\gamma_n)(L)$ to the leaf space are contained in the image of $D$. 
Let $A_n$ denote the projection of $J_n$ to $D$.  After passing to a subsequence, the arcs $A_n$ converge, in the Hausdorff metric, to some nondegenerate set $A \subset D$ that lies in the leaf $L_\infty$.   
By continuity of the endpoint map, the image of $\overline{e}_\rho$ on $A$ contains a connected set of the form $\{q\} \times J$, where $J$ is a nondgenerate segment.  

However, as in Lemma \ref{lem:local_transversal}, we may choose the transversal $D$ so that the restriction of this transversal to each horizontal leaf $L'$ is the parametrization given by the negative endpoint map $e_\rho^-$.   This contradicts that we have a nondegenerate subset of $L_\infty$ in $D$ mapping to $\{q\} \times J$, with negative endpoint constant.   
\end{proof}

\subsubsection*{Conclusion of proof of Theorem \ref{thm:local_rig}}
We have just shown that, for representations $\rho$ in some neighborhood of $\rho_0$, the endpoint map $e^+_\rho$ is constant on each set $\tilde{M} \times \{p\} \subset \tilde{M} \times S^{n-1}$, so descends to a continuous map $S^{n-1} \to S^{n-1}$.  Lemma \ref{lem:surjective} implies that this map is surjective, and by construction, we have 
\[ e^+_\rho  \rho(\gamma)(x) = \rho_0(\gamma)   e^+_\rho(x) \]
as in Equation \eqref{eq:equivariance}, so $e^+_\rho$ is the desired semi-conjugacy between $\rho$ and the standard boundary action of $\pi_1M$.  
Strong topological stability (the claimed control on the semi-conjugacy) follows from continuity of $\rho \mapsto e^\pm_\rho$ proved in Lemma \ref{lem:cont_rho}.  
  \qed

\subsection{Topological stability of geodesic flows} \label{sec:KatoMorimoto}
We conclude this section with a short sketch of how our proof above gives a ``soft'' geometric proof of topological stability of the geodesic flow in negative curvature, as claimed in Theorem \ref{cor:straightening}.  

\begin{proof}[Proof of Theorem \ref{cor:straightening}]
Let $M$ be a closed manifold of negative curvature and $\Phi_t$ the geodesic flow on $\UT M$. Suppose that $\Psi_t$ is a flow such that the flowlines of the lift $\tilde \Psi_t$ to $\UT\tilde{M}$ each $\epsilon$ fellow-travel flowlines of $\tilde\Phi_t$ on segments of length $R$, as in the statement of Theorem \ref{cor:straightening}.  
The local-to-global principle (Lemma \ref{lem:local_to_global}) implies that there exists $N$ and $c$ such that, if $R \geq N$ and $\epsilon \leq c$, then flowlines of $\tilde\Psi_t$ project to quasi-geodesics in $\tilde{M}$, and so each flowline of $\tilde\Psi_t$ stays within a bounded distance of a unique flowline of $\tilde \Phi_t$, and so has well defined endpoints.   Lemma \ref{lem:fellow_travel} implies that if $\epsilon$ is sufficiently small, as $R \to \infty$ these endpoint maps $e^{\pm}: \tilde{M} \to \partial_\infty\tilde{M} \times \partial_\infty\tilde{M}$, sending a point $x$ to the positive and negative endpoints of the flowline $\tilde\Psi_t(x)$ converge uniformly on compact sets to the endpoint map for the geodesic flow.    By construction $e^{\pm}$ are $\pi_1(M)$-equivariant, so by the same argument as in Lemma \ref{lem:surjective}, we may conclude that $e^\pm$ is surjective onto the complement of the diagonal in $\partial_\infty\tilde{M} \times \partial_\infty\tilde{M}$, which is naturally identified with the flow space of $\tilde \Psi_t$.  This gives a $\pi_1M$-equivariant, continuous, surjective map from $\tilde{M}$ to the flow space of $\tilde \Psi_t$, which descends to a map defined on the flowspace of $\tilde \Phi_t$.

To improve this map on the level of orbit spaces to a topological equivalence of the flows, one may now use the averaging trick in Barbot \cite[Theorem 3.4]{Barbot} following Ghys \cite[Lemmas 4.3, 4.4]{Ghys84}.  Specifically, 
define first a map $h_0$, associating to each point $x \in \tilde M$ the closest point to $x$ on the geodesic between $e^+(x)$ and $e^-(x)$.  This maps flowlines to flowlines, but may not send a flowline injectively onto its image.  Rather, there is simply a continuous function $a: \R \times \tilde{M}$ satisfying $h_0(\tilde \Phi_t(x)) = \tilde \Psi_{a(t,x)}(h_0(x))$.   
To remedy this, fix $T$ large, and define $A(t) = \tfrac{1}{T}\int_0^T a(s, x) ds$. 
One checks that, if $T$ was chosen sufficiently large, the map 
\[ h(x):= \tilde \Phi_{A(t)}(h_0(x)) \]
sends each flowline of $\tilde \Psi_t$ continuously and injectively onto a flowline of $\tilde \Phi_t$, and descends to a continuous map $M \to M$ giving a topological equivalence of the flows. 
\end{proof}


\section{Examples}  \label{sec:blowup}
In this section we illustrate some of the phenomena that can appear in Theorem \ref{thm:local_rig}. We give two families of examples of actions that are semi-conjugate, but not conjugate, to the action of the fundamental group of a closed negatively curved manifold on its boundary. The first uses the work of Cannon and Thurston, and is specific to Kleinian groups. The second extends the classical Denjoy blow-up and applies to any action of regularity $C^1$. 
\subsection*{Cannon-Thurston Maps} 
We briefly summarize the construction of the Cannon--Thurston map (in a special case), following \cite{CT}.
Let $S$ be a closed, hyperbolic surface, $\phi$ a pseudo-Anosov diffeomorphism, and $M$ a hyperbolic 3-manifold given by the suspension of $\phi$, equipped with the suspension flow $\varphi_t$ of the pseudo-Anosov map $\phi$.   Lifting flowlines to the universal cover $\widetilde{M} = \H^3 $ gives a flow $\tilde{\varphi}_t$ whose flow space is a topological disk $D$, which may be identified with the universal cover $\tilde{S} \subset \tilde{M}$ of any fiber $S$ of $M$.  It is easily verified that flow lines of $\tilde{\varphi}_t$ are quasi-geodesics in $\H^3$, so we have continuous endpoint maps $e_{\pm}: D \to \partial_\infty \H^3$.  

Identifying $D$ with $\tilde{S}$, we have the standard boundary compactification $\widehat{D}= \tilde{S} \cup \partial_\infty\tilde{S}$.  
Cannon and Thurston \cite{CT} showed the action of $\pi_1M$ extends to the closed disk $\widehat{D}$ in a way that is compatible with the positive and negative endpoint maps.  This gives maps 
$$\hat{e}_\pm: \widehat{D} \longrightarrow \partial_\infty \H^3.$$
 These extensions coincide on the boundary $\partial_\infty\tilde{S}$ and are $\pi_1M$-equivariant. Gluing these together along the boundary, we obtain a $\pi_1M$-equivariant map
$$h_{CT} = \hat{e}_- \cup \hat{e}_+ : S^2 =( \widehat{D}_- \cup_{S^1} \widehat{D}_+)  \longrightarrow \partial_\infty \H^3.$$ 
This gives an induced action $\rho_{CT}$ of $\pi_1M$ on $S^2$. By equivariance of the construction and by minimality of the action of $\pi_1M$ on $ \partial_\infty \H^3$, we conclude that $h_{CT}$ is surjective.  Additionally, it follows directly from the construction that preimages of points under $h_{CT}$ are either points, closures of complementary regions of the stable or unstable geodesic lamination of $\varphi$, or closures of geodesics in $\widehat{D}$.  In particular $h_{CT}$ has contractible point-preimages and hence, by Moore \cite{Moore}, it can be approximated by homeomorphisms.  Let $h_n \in \Homeo_+(S^2)$ be a sequence of homeomorphisms such that $h_n \longrightarrow h_{CT}$ in the compact open topology in the space of continuous maps $S^2 \to S^2$. Then the conjugate actions
$h_n \circ \rho_{CT} \circ h^{-1}_n$
converge in the weak sense (element-wise) to the boundary action.   

In other words, in any neighborhood of the boundary action, there are conjugates of $\rho_{CT}$.  Note that none of these are themselves conjugate to the boundary action, as $\rho_{CT}$ is not minimal -- it has an invariant circle.  
We note also that $\rho_{CT}$ itself (and hence any conjugate of it) is rather flexible: the Alexander trick allows one to produce a continuous deformation from $\rho_{CT}$ to an action of $\pi_1M$ on $S^2$ with a global fixed point by continuously shrinking one hemisphere while enlarging the other.  

While we have described this construction for fibered hyperbolic 3-manifolds, it applies more broadly:  work of Frankel \cite{Frankel} shows that the Cannon--Thurston construction can be modified to give an analogous map on any closed hyperbolic manifold admitting a {\em quasi-geodesic} flow.

\subsection*{A ``blow-up" example.}  We describe how to equivariantly blow up an orbit $\Gamma \cdot{z}$ of a $C^1$ action of a countable group on an $n$-sphere to produce an action by homeomorphisms that is semi-conjugate to the original.  The semi-conjugacy map $h$ will be injective off of the preimage of this orbit, and have the additional property that preimages of points in $\Gamma \cdot{z}$ are homeomorphic to closed disks. In particular, $h$ may be approximated by homeomorphisms.
 
While our intended application is boundary actions of manifolds admitting negatively curved metrics, the construction applies quite generally to any $C^1$ action of a countable group on $S^n$ so we work in this broader context.   For actions on $S^1$ a similar construction works even for actions by homeomorphisms, and can, at least for abelian groups be smoothed to a $C^1$ action; this is the classical Denjoy blow-up.   The construction below could conceivably be generalized to group actions on any manifolds, however ensuring that the space obtained by ``blowing up" an orbit is again a manifold requires some care; here we are able to quote Cannon's description of Sierpinski spaces.
\begin{proposition}  \label{prop:blowup}
Let $\Gamma$ be a countable group and $\rho: \Gamma \to \Diff_+^1(S^{n})$ an action with dense orbit $Z$.  
Then there exists $\rho': \Gamma \to \Homeo_+(S^n)$ and a surjective, continuous map $h: S^n \to S^n$, such that the pre-image of each point in $Z$ is a closed disk, that is injective on the complement of $h^{-1}(Z)$, and such that $h \circ \rho = \rho' \circ h$.  
\end{proposition} 

While blowing up a {\em finite} orbit under a group action is a standard construction, we know of no reference in the literature (beyond that for actions on $S^1$) for this result, so we give a proof.  

\begin{proof}
Our strategy is to use an inverse limit construction.  For simplicity, we assume that $Z$ is the orbit of a point $z$ with trivial stabilizer, however the construction works more generally using the fact the point stabilzers act naturally on the tangent space at any fixed point.   
Enumerate $\Gamma = \{ \gamma_1, \gamma_2, \ldots \}$, and let $z_n = \gamma_n(z)$.  
Let $X_0$ denote the unit sphere $S^n$ with the standard round metric.  Fix some small $\epsilon_1>0$, let $D'_1$ denote the $\epsilon_1$-ball about $z_1$, and $D_1 \subset D'_1$ the $\frac{1}{2}\epsilon_1$ ball about $z_1$.   Let $X_1 = X_0 - D_1$ and define $f_1: X_1 \to X_0$ to be a $C^\infty$ map that is the identity on $X_1 - D'_1$ and is a radial collapse along geodesics through $z_1$ on the annulus $(D'_1 - D_1) \subset X_1$ that sends $\partial D_1$ to the point $z_1$ and is injective otherwise.
In this way $f_1^{-1}$ gives an identification of $\partial D_1$ with the positive projectivized tangent space of oriented lines at $z_1$, so that the action of any $C^1$ diffeomorphism $g$ of $S^n$ fixing $z_1$ naturally extends to a homeomorphism $\hat{g}$ of $X_1$ such that $f_1\circ \hat{g} = g \circ f_1$.   Additionally, we can ensure this map has the Lipshitz property that $d(f_1(x), f_1(y)) < 3 d(x,y)$ for all $x, y \in X_1$, or equivalently, $d(f_1^{-1}(x), f_1^{-1}(y)) > \frac{1}{3} d(x,y)$.

Now inductively, suppose that for all $m \leq k$ we have defined $X_m \subset S^n$ (topologically, a sphere with $m$ holes) and a $C^\infty$ surjective map $f_m: X_m \to X_{m-1}$. Let $F_m : X_m \to X_0$ denote the composition $f_m f_{m-1} \ldots f_1$.  Choose some $\epsilon_{k+1}  \le \epsilon_k/2$ that is additionally less than half the distance from $F_k^{-1}(z_k)$ to the nearest boundary component of $X_k$.  Choose $r_{k+1} \ll 1$, and define $X_{k+1}$ to be $X_k$ with an open $r_{k+1} \epsilon_{k+1}$-ball about $F_{k+1}^{-1}(z_{k+1})$ removed, and $f_{k+1}: X_{k+1} \to X_k$ a map that collapses the boundary of the removed disk to the point $F_{k+1}^{-1}(z_{k+1})$, with support on a $\epsilon_{k+1}$ disk, defined using the same procedure as above.  If $r_{k+1}$ is chosen sufficiently small, then we can ensure that this map has the Lipschitz property 
\[ d(f_{k+1}^{-1}(x), f_{k+1}^{-1}(y)) > \lambda_{k+1} d(x,y) \]
for all $x$ and $y$ (and all choices of points in the preimages in the degenerate case where $x = y = z_{k+1}$), moreover $\lambda_{k+1} < 1$ can be taken as close to 1 as we like, by choosing $r_{k+1}$ close to 0.  Make such a choice, inductively, so that the product $\lambda_1 \lambda_2 \lambda_3 \ldots$ converges to some $\delta > 0$.   
 
As before, the induced identification of the tangent space at $z_{k+1}$ with the boundary of the disc $D_{k+1}$ means that  any diffeomorphism of $X_k$ fixing $z_k$ defines a diffeomorphism of $X_{k+1}$.  More generally, if $g$ is a diffeomorphism of $S^n$ that preserves the set $\{z_1, \ldots z_{k+1}\}$, it also defines a diffeomorphism of $X_{k+1}$ (via conjugation by $F_{k+1}$ on the invariant set $S^n - \{z_1, \ldots z_{k+1}\}$ on which $F_{k+1}^{-1}$ is a diffeomorphism, and on the inserted boundary disks by the identification of them with the tangent spaces to the points $z_i$).  

To summarize, these spaces and maps have the following properties:
\begin{enumerate}
\item[(i)] The sets $X_k \subseteq X_{k-1} \subseteq S^n$ form a monotone decreasing family;
\item[(ii)] Under the map $f_k: X_k \to X_{k-1}$, each point has a unique preimage, except for the point $z'_k =F_{k-1}^{-1}(z_k)$ whose preimage is the boundary of the $r_k \epsilon_k$-ball $D_k$ about $z'_k$. 
\item[(iii)] For each $f_k$ we have $\displaystyle{\sup_{x\in X_k} d(f_k(x),x) \le \epsilon_{k} \le \frac{\epsilon_1}{2^{k-1}} < \frac{1}{2^{k-1}} }$.
\item[(iv)] For any $x, y \in X_0$ we have $d(F_k^{-1}(x),  F_k^{-1}(y)) > \delta d(x,y)$ for all $k \in \N$.

\end{enumerate} 

Let $X=\varprojlim X_k$ be the inverse limit of the system of maps $X_k \overset{f_k}\to X_{k-1}$, i.e. 
\[ X := \{ (\ldots ,p_2, p_1, p_0) \ | \ p_i \in X_i \textrm{ and } p_{k-1} = f_k(p_{k})  \}. \]
Since each map $f_k$ is continuous, $X$ is a closed subset of the product $\prod_{k=0}^\infty X_k$, hence is also compact.
Property (iii) above means that for each element $(\ldots ,p_2, p_1, p_0) \in X$, the points $p_0, p_1, \ldots $ form a Cauchy sequence in $S^n$ so $\lim_{k \to \infty} p_k$ is well defined; since $p_k \in X_k \subset X_{k-1}$ the limit lies in the intersection $\cap_k X_k$.   Define $\phi: X \to \cap_k X_k \subset S^n$ by setting $\phi(\ldots ,p_2, p_1, p_0) = \lim_{k \to \infty} p_k$.   
Property (iv) above ensures that $\phi$ maps sequences associated to distinct points to distinct limits, so $\phi$ is injective. Since $X$ is compact and $S^n$ is Hausdorff, $\phi$ is therefore a homeomorphism onto its image.    
Note that the image of $\phi$ contains the union of all boundaries of removed discs $D_k$, since for any $p_k \in \partial D_k$, the sequence $(\ldots p_k, p_k, p_k, f_k(p_k), f_{k-1} f_k(p_k) \ldots F_k(p_k))$ is an element of $X$, with $F_k(p_k) \in Z$. 

We will now make use of the following result of Cannon.  
\begin{theorem}[Cannon \cite{Cannon73}] \label{thm:Cannon}
Let $S \subset S^{n}$ be a closed subset, and let $U_i$ denote the connected components of $S^n-S$.  Then $S$ is homeomorphic to the (unique up to homeomorphism) $n-1$ dimensional Sierpinski space if and only if the following hold
\begin{enumerate}
\item For each $i$, $S^n-U_i$ is an $n$-cell
\item The closures of the $U_i$ are pairwise disjoint
\item $\bigcup_i{U_i}$ is dense in $S^n$, and 
\item $U_1, U_2, \ldots$ is a null sequence, meaning that $\mathrm{diam}(U_n) \to 0$.
\end{enumerate} 
\end{theorem}
\noindent Cannon's result is stated for $n \neq 4$, but applies in all dimensions given Quinn's proof of the Annulus theorem in dimension 4.  
Apply this to the set $S = \cap_n X_n$.  The complimentary regions are the discs $D_i$.  By construction, they have pairwise disjoint closures, $S^n - D_i$ is an $n$-cell, and the sets form a null sequence.   To see that $\bigcup_i{D_i}$ is dense, suppose for contradiction that some open ball $B$ of radius $\epsilon > 2^{-k}$ was in the complement of the closure of $\bigcup_i{D_i}$.   Consider the sequence of maps $f_{k+m} \circ \ldots f_{k+2}  \circ f_{k+1}$ defined on $\cap_n X_n$, for $k$ fixed, as $m \to \infty$.  By property (iii) above, these pointwise converge to a map $\cap_n X_n \to X_k$ which moves all points distance at most $2^{-k-1}$.  Thus, the image of $B$ under the limit contains an open set in $X_k$.  However, such a set must intersect the dense set $F_k^{-1}(Z)$, contradicting the fact that $B$ does not intersect the closure of the union of the discs $D_i$.   We conclude that $S$ is a Sierpinski space.   

Since $\phi$ is a homeomorphism from the compact space $X$ whose image contains a (closed) dense subset of the Sierpinski space $\cap_n X_n$, we conclude that $\phi(X) = \cap_n X_n$.   
We claim now that the projection $F: X \to X_0 = S^n, (\ldots ,p_2, p_1, p_0)  \mapsto p_0$ 
gives a homeomorphism between the set $X - F^{-1}(Z)$ and $S^n-Z$ and that there is an action of $\Gamma$ on $X$ by homeomorphisms such that the restriction to $X - F^{-1}(Z)$ agrees (under this homeomorphic identification) with the original action of $\Gamma$. Given this, collapsing each boundary of a connected component of the compliment of $X \cong \phi(X)$ to a single point gives a sphere $\overline{X}$, and $F$ induces a continuous, surjective map $\overline{X} \to S^n$ that intertwines the two actions, as desired.  
As we have already observed the restriction of $F$ to $X - F^{-1}(Z)$ is injective, which implies that $F: X - F^{-1}(Z) \to S^n - Z$ is a homeomorphism, since it is a continuous bijection induced from the map $F$, and $F$ is a continuous map between compact metric spaces, hence closed.

The action of $\Gamma$ on $X$ comes from our description of the $X_i$ as a blow-up of the tangent space to a point.  For each $i$, the set $F^{-1}(z_i)$ is a circle, identified with the projectivized tangent space (the space of oriented lines in $T_{z_i}(S^n)$), via projection to $X_i$ and the identification there.  This gives an action of $\Gamma$ by bijections of $X$; it remains to show that it is in fact an action by homeomorphisms.  For this it suffices to check continuity of the action of each $\gamma \in \Gamma$.    Let $x_n \to x_\infty$ be a sequence of points in $X$.  If $x_\infty \notin F^{-1}(Z)$, that $\gamma(x_n)$ converges to $\gamma(x_\infty)$ follows directly from our construction and the definition of the inverse limit.  
If $x_\infty \in F^{-1}(z_j)$ for some $z_j$, then it suffices to project to $X_j$ and work there. That $x_n$ converges to $x_\infty$ in $X_j$, where $x_\infty$ is a boundary point means precisely that, as $n \to \infty$, the points $F_j(x_n)$ converge to $z_j$ and $\frac{F_j(x_n)-z_j}{|| F_j(x_n)-z_j||}$  converges to the tangent direction $v$ represented by $z_\infty$.  Continuous differentiability of $\gamma$ at $z_j$ is all that is required to have $\gamma(x_k) \to \gamma(x_\infty)$, this is why we assumed our original action was of class $C^1$. 
\end{proof}

\section{Global rigidity of slitherings from skew-Anosov foliations} \label{sec:skew}

In this section we specialize to actions of fundamental groups of certain 3-manifolds on $S^1$.  In this case, Lemma \ref{lem:leafwise_imm} gives a homeomorphism rather than a continuous map, and we will exploit this property to prove a global rather than local rigidity result for (lifts of) boundary actions and the more general case of actions induced by ``slitherings'' from skew-Anosov flows.  We begin by summarizing some standard results and framework needed for the proof.  

\subsection{Anosov Flows}
A flow $\Phi_t$ generated by a vector field $Y$ on a closed 3-manifold $M$ is {\em Anosov} if the tangent bundle splits as a sum of (continuous) line bundles that are invariant under the flow
$$TM =   E^{ss} \oplus \langle Y \rangle \oplus E^{uu}$$
with the property that for some choice of metric on $M$, there are constants $C, \lambda >0$ such that
$$||(\phi^t)_*(v^{s})|| \leq C e^{-\lambda t} ||v^{s}|| \text{ and }  ||(\phi^t)_*(v^{u})|| \geq C^{-1} e^{\lambda t} ||v^{u}||$$
holds for all $ t \ge0$ and all $v^{u}\in E^{uu}, v^{s} \in E^{ss}$. 
By averaging the metric over long time intervals and decreasing $\lambda$, one can assume that $C = 1$. Such a metric is called {\em adapted}.

The line fields $E^{uu}, E^{ss}$ are called the {\em strong unstable} and {\em strong stable} directions of the flow. It is a classical fact that these distributions are uniquely integrable.  The foliations to which they are tangent are characterized by the dynamical property that their leaves consist of sets of points that are asymptotic under the flow in forward, respectively backward, time. 
One also obtains foliations $\F^s$ and $\F^u$ tangent to the integrable plane fields 
$$ E^{s}= E^{ss} \oplus \langle Y \rangle \ , \  E^{u} = E^{uu} \oplus \langle Y \rangle$$
respectively; these are called the {\em weak stable} and {\em unstable} 
foliations of the flow.  In the examples of interest to us the line fields $E^{ss}$ and $E^{uu}$ will always be orientable, i.e.\ trivial as line bundles, so from now on we take orientability to be a standing assumption.  

The following proposition collects some well-known properties of the weak foliations of an Anosov flow that we will need going forward.  
The additional $C^1$ structure given by point (1) below will be important in the proof of Theorem \ref{thm:slith_rigid}.

\begin{proposition}\label{prop:Anosov_properties}
Let $\Phi_t$ be an Anosov flow on a closed $3$-manifold $M$. Then the following hold
\begin{enumerate} 
\item[(i)] (Hirsch-Pugh \cite{Hirsch_Pugh}): The weak stable and unstable foliations $\F^s$ and $\F^u$ are of class $C^1$.
\item[(ii)] $M$ admits a metric such that the induced metric on weak stable and unstable leaves in $\tilde{M}$ is uniformly bi-Lipschitz equivalent to a metric of constant  curvature $-1$.  In this metric, the flowlines on each leaf are quasi-geodesics;  on a leaf of $\F^s$, flowlines share a unique common forward endpoint, and on $\F^u$ a common negative endpoint.  
\end{enumerate}
\end{proposition}

For completeness, we give an outline of the proof.  The reader may consult \cite[Section 5]{Fenley} for more details and general background.  

\begin{proof}
Item (1) follows from the proof of the Smoothness Theorem part (i) in \cite{Hirsch_Pugh}.  Specifically, one applies the graph transform argument there to the quotient bundle $TM / \langle Y \rangle$ upon which the flow acts. This action has two invariant sub-bundles 
$\overline{E}^{s} , \overline{E}^{u}$ given by the images of the weak stable and unstable subbundles. Since these are uniformly contracted and expanded, respectively, by $D \Phi_t$, the $C^1$-section Theorem \cite{HPS}
then implies that $\overline{E}^{s} $ and $\overline{E}^{u}$ are $C^1$. Pulling back to $TM$, one deduces that the subbundles $ E^{s}, E^{u}$ are of class $C^1$ as well.  Since they are invariant under the flow, it follows that they are tangent to $C^1$-foliations.

To show item (2), take a $C^0$-metric on $M$ so that the strong stable/unstable directions and the flow direction are all orthogonal, and the generating vector field has unit length. Without loss of generality we assume that this metric is adapted to the Anosov flow.  In general, this metric may only be continuous, but we do not need any higher regularity for the argument.  
Let $L$ be a leaf of the weak unstable foliation, and $\ell_s$ a strong-stable leaf through some point $p \in L$.  Then $\ell_s$ is a section for the restriction of $\Phi_t$ to $L$.  Parametrize $\ell_s$ by arc length and call this coordinate $x$.  The lift $\widetilde{\ell}_s$ gives a section for the induced flow on the universal cover $\widetilde{L}$ and hence a global coordinate system $(x,t)$ on $\widetilde{L}$ so that the pulled-back metric is of the form
$$\begin{pmatrix}
    f^2(x,t)        & 0\\
    0         & 1
\end{pmatrix}$$
In particular, the flow lines are geodesics with respect to this metric. By construction $f(x,0) = 1$ and the Anosov condition gives the bounds 
$$ \varepsilon e^{-\lambda t} \le f(x,t) \le e^{-\lambda t}.$$
This implies that the the metric on $\widetilde{L}$ is uniformly bi-Lipshitz equivalent to the pull-back of the flat metric on $\R^2$ by $(x,t) \mapsto e^{-\lambda t}$, i.e. a constant negative curvature hyperbolic metric on the upper half plane, hence bi-Lipshitz equivalent to standard hyperbolic metric of constant curvature -1.   In the hyperbolic metric, vertical lines are geodesics with the same forward endpoint, and these correspond to flowlines under our bi-Lipshitz identification.  
The case of the unstable foliation follows {\em mutatis mutandis}.
\end{proof}

\subsection{Slitherings and skew-Anosov flows}
We first recall the notion of a slithering, as introduced by Thurston in \cite{ThurstonSlithering}. 
\begin{definition}[Slithering]
Let $M$ be a closed 3-manifold.  A {\em slithering} of $M$ over $S^1$ is a fibration $s: \widetilde{M} \to S^1$ with $2$-dimensional fibers such that deck transformations are bundle automorphisms for $s$, taking fibers to fibers.  
This means that the foliation of $\widetilde{M}$ given by the fibers of $s$ descends to a foliation on $M$.
\end{definition}

Since deck transformations take fibers to fibers, a slithering $s: \widetilde{M} \to S^1$ also induces a natural {\em slithering action}  $\rho_s: \pi_1M \to \Homeo_+(S^1)$ on the circle.  Following our earlier convention, we continue to  assume that all foliations are oriented, so this slithering action is by orientation preserving homeomorphisms.  
Slitherings generalize both the notion of a fibering over $S^1$ (where $s$ is simply the lift of the bundle projection to $\widetilde{M}$), and the notion of a foliated $S^1$-bundle, where $s$ is the projection to the fiber on the induced foliated bundle over $\widetilde{M}$.   Skew-Anosov flows (a generalization of geodesic flows on negatively curved surfaces) provide another important source of examples.

\begin{example}[Skew-Anosov flows]\label{ex:skew_Anosov}
Let $\Phi_t$ be an Anosov flow on a closed 3-manifold $M$, whose stable foliation is oriented and $\mathbb{R}$-covered, meaning that the leaf space on the universal cover is Hausdorff (or equivalently, is homeomorphic to $\mathbb{R}$).  Results of Fenley \cite{Fenley} and Barbot \cite{Barbot} show that a flow with this property is either the suspension of an Anosov diffeomorphism of $T^2$ or is {\em skew}, meaning that the orbit space of the lift of the flow to $\widetilde{M}$ is homeomorphic to the infinite diagonal strip 
\[\mathcal{O} = \{ (x,y) \in \mathbb{R}^2 \ | \ |x - y| <1\}\] in such a way that the preimages of horizontal (respectively, vertical) intervals are the stable (resp. unstable) leaves of the flow, as illustrated in Figure \ref{Skew_Flow}.
\begin{figure}
  \labellist 
  \hair 2pt
     \pinlabel $o_u$ at 135 160
    \pinlabel $o_l$ at 215 95
   \endlabellist
   \centerline{ \mbox{
 \includegraphics[width = 2.5in]{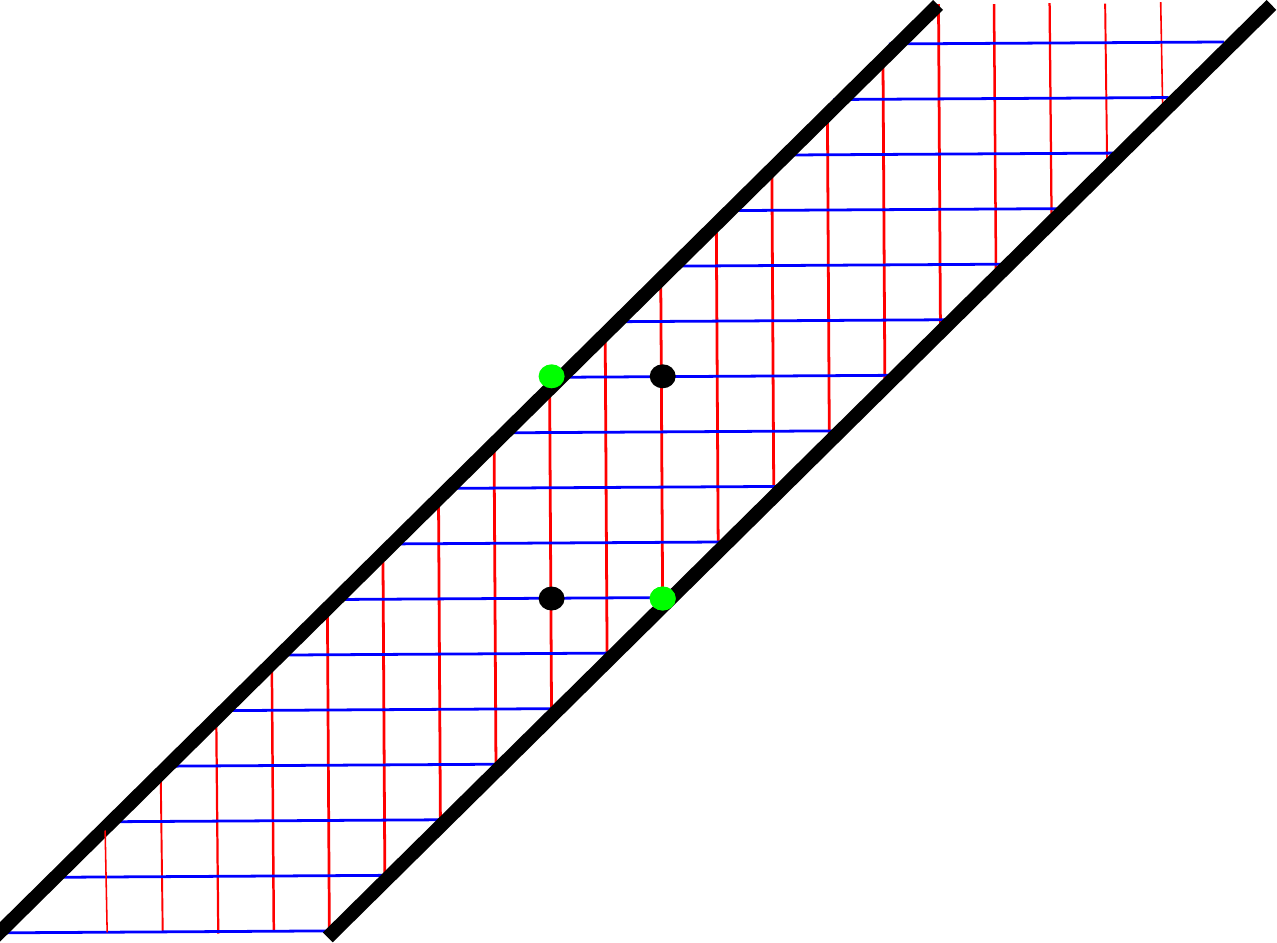}}}
 \caption{The flow space $\mathcal{O}$ of a skew-Anosov flow. The two black points are related by the natural map $\eta$ on  $\mathcal{O}$}
  \label{Skew_Flow}
\end{figure}

In this model, each point $o \in \mathcal{O}$ can be assigned a point $o_u$ on the upper boundary by following the unstable leaf through $o$, and a point $o_l$ on the lower boundary by following the unstable leaf.  Taking the intersection of the stable leaf through $o_u$ and unstable through $o_l$ defines a continuous, fixed point free map $\eta: \mathcal{O} \to \mathcal{O}$.  This map sends stable leaves to unstable leaves and vice versa, so $\tau = \eta^2$ descends to a map on the leaf space $\Lambda^s$ of the weak stable foliation.  This map is strictly monotone, and the quotient map.
\[\tilde{M} \to \Lambda^s \to \Lambda^s/\tau\] defines a slithering of $M$.
By construction, the foliation associated to this slithering is the weak stable foliation of $\Phi_t$.
\end{example}

The map $\eta$ has many remarkable properties.  A concise summary is given in \cite[\S 4]{BF}; we will simply state those which are of use to us.  First, $\eta$ is a $\pi_1M$-equivariant homeomorphism and can be induced from a continuous self-map $\eta_M$ of the underlying manifold $M$ \cite{Barbot,Fenley}.   Barbot \cite[Theorem 3.4]{Barbot} showed, using an averaging argument, that this map $\eta_M$ can actually be taken to be a homeomorphism of $M$.   An alternative description of $\eta_M$ map is given in \cite[Prop 7.4 ii)]{ThurstonSlithering}.  
Futhermore, if some element of the fundamental group fixes a point $o$ of the leaf space, then it also fixes $\eta^k (o)$ for all $k$, and the corresponding periodic orbits of the flow are freely homotopic. It is not hard to see that the converse is also true: any two periodic orbits of a skew-Anosov flow that are freely homotopic are related by some power of the map $\eta$ on the flow space. We note this fact for later use.
\begin{proposition}[see \cite{Barbot,Fenley}]\label{prop:freely_homotopic}
Let $\alpha,\beta$ be freely homotopic orbits of a skew-Anosov flow with orientable splitting on a closed manifold $M$. Then $\beta = \eta_M^k (\alpha)$ for some integer $k$, where $\eta_M$ is a homeomorphism of $M$ that induces the map $\eta$ on the flow space. 
\end{proposition}

We will also need to use the following result of Barbot on minimality of the slithering action associated to a skew-Anosov flow. 
\begin{proposition}[\cite{Barbot} Theorem 2.5]\label{prop:skew_Anosov_minimal}
Any skew-Anosov flow is transitive and its associated slithering action  $\rho_s: \pi_1M \to \Homeo_+(S^1)$ is minimal.
\end{proposition}

\subsection*{Universal circles}
Let $\F^u$ and $\F^s$ denote the lifts to $\tilde{M}$ of the unstable and stable foliations of an Anosov flow.   Following Proposition \ref{prop:Anosov_properties} ii), the leaves of these foliations have a natural large-scale hyperbolic structure, hence can be compactified by a boundary at infinity.  In the case of a skew-Anosov flow, Thurston \cite{ThurstonSlithering} observed that these foliations are {\em uniform}, meaning that any pair of leaves is lie a bounded distance apart from each other, and hence the leafwise boundaries can be canonically identified: 

\begin{lemma}[Lemma 4.1 and Corollary 4.2 of \cite{ThurstonSlithering}] \label{lem:uniform_foliation}
For each pair of leaves $L$ and $L'$ of $\F^u$, and every infinite geodesic $g$ on L, there is a unique geodesic $g'$ on $L'$ at a a bounded distance from $g$.  This produces a canonical identification of the circles at infinity for all the leaves of $\F^u$.  We call the result the {\em universal circle} at infinity.   The same holds with $\F^s$ in the place of $\F^u$. 
\end{lemma}

An alternative way to describe Thurston's universal circle is by considering the intersection of $\F^u$ and $\F^s$.  For a fixed leaf $L$ of $\F^u$, the leaves of $\F^s$ intersect $L$ as quasi-geodesics with a common forward (using an induced orientation) endpoint, with respect to the large-scale hyperbolic structure.  Thus, the boundary of $L$, minus one point, can be identified with a subset of the leaf space of $\F^s$, and there is a natural map defined on subsets of boundaries of any two nearby leaves $L$ and $L'$ of $\F^u$ via leaves of $\F^s$.   This gives the following. 

\begin{proposition}[Prop 7.1 of \cite{ThurstonSlithering}]  \label{prop:universal_circle}
Let $S^1_u$ denote the universal circle obtained from Lemma \ref{lem:uniform_foliation}.  There is an identification of $S^1_u$ with $\Lambda^s/\tau$, under which the action of $\pi_1M$ on $S^1_u$ (i.e. as obtained from the action on the leaf space) agrees up to conjugacy with the slithering action of the foliation. 

\end{proposition} 
In other words, $S^1_u$ can be thought of as the space of vertical lines (mod $\tau$) of the orbit space $\mathcal{O}$ depicted in Figure \ref{Skew_Flow}.

\subsection{Proof of Theorem \ref{thm:slith_rigid}}
In this section we use the following notion of semi-conjugacy for circle maps, as defined by Ghys in \cite{Ghys87}.   Though the terminology ``semi-conjugacy'' is now widespread, this is not the same as the standard dynamical notion of semi-conjugacy defined in Section \ref{sec:prelim}.  
To avoid confusion, we will follow \cite{MannWolff} and use the term {\em weak conjugacy} for Ghys' definition.  
\begin{definition} \label{def:semiconj}
Let $\rho_1$ and $\rho_2: \Gamma \to \Homeo_+(S^1)$ be two actions of a group $\Gamma$ on the circle $S^1 = \R/\Z$.  These actions are {\em weakly conjugate} if there is a monotone map $h: \R \to \R$ commuting with $x \mapsto x+1$, and lifts of each element $\rho_i(\gamma)$ to $\Homeo_+(\R)$ satisfying $h \circ \widetilde{\rho_1(\gamma)} = \widetilde{\rho_2(\gamma)} \circ h$. 
\end{definition}

The map $h$ in the definition above is not required to be continuous or surjective.  However, if $\rho_2$ is minimal, any weak conjugacy $h$ between $\rho_2$ and any other representation $\rho_1$ is necessarily continuous and surjective.   Note that, since $h$ commutes with integer translations, it descends to a map of $S^1$.  A map of $S^1$ so induced is called a {\em degree one monotone map}.  It is easy to verify that the surjective, degree one monotone maps of $S^1$ are precisely the orientation-preserving maps of $S^1$ which are approximable by homeomorphisms.

We divide the proof of Theorem \ref{thm:slith_rigid} into two propositions, covering first the local then the global result.  
\begin{proposition}[Local Rigidity]\label{prop:slith_rigid_local}
Let $\F^s$ be the weak stable foliation of a skew-Anosov flow $\Phi_t$ on $M^3$ with associated slithering action $\rho_s: \pi_1M \rightarrow \Homeo_+(S^1)$.  Then there exists a neighborhood $U$ of $\rho_s$ in $\Hom(\pi_1M, \Homeo_+(S^1))$ consisting of representations weakly conjugate to $\rho_s$.  
\end{proposition}

\begin{proposition}[Global Rigidity]\label{prop:slith_rigid_global}
Under the hypotheses above, one can in fact take $U$ to be the connected component of $\rho_s$ in $\Hom(\pi_1M, \Homeo_+(S^1))$.
\end{proposition}

The proof of the local version follows roughly the same strategy as that of Theorem \ref{thm:local_rig} in Section \ref{sec:boundary_rig}.  However, here the suspension of the action is one dimension larger than that considered there, forcing us to make use of a natural section in order to cut down a dimension.  The proof of the global result is then a quick consequence of approximability of weak conjugacy maps by homeomorphisms.  

\begin{proof}[Proof of Proposition \ref{prop:slith_rigid_local}]

Let $\F^s$ be the weak stable foliation of a skew-Anosov flow $\Phi_t$ on a closed 3-manifold $M$,  let $s: \widetilde{M} \to S^1$ be the associated slithering, and let $\rho_s: \pi_1M \to \Homeo_+(S^1)$ be the slithering action.  For clarity, we divide the proof into steps, as indicated by the paragraph headings.  

\paragraph{Setup: a canonical section}
Consider the lift of $\F^s$ to $\tilde{M}$.  As in Section \ref{sec:boundary_rig}, we abuse notation slightly and let $\F^s$ also denote the lifted foliation to $\tilde{M}$, with leaf space $\Lambda^s \cong \R$.  
By Proposition \ref{prop:Anosov_properties}, $\F^s$ is of class $C^1$, giving a $C^1$ identification $\tilde{M} \cong \R^2 \times \Lambda^s$, and an action of $\pi_1M$ on $\Lambda^s$ by $C^1$ diffeomorphisms.  As explained in our discussion earlier (see Example \ref{ex:skew_Anosov}), this action commutes with the map $\tau: \Lambda^s \to \Lambda^s$ used in defining the slithering, and $\rho_s$ is simply the induced action of $\pi_1M$ on the (topological) circle $\Lambda^s/\tau$.   Note, however, that the map $\tau$ is in general only a homeomorphism and so $\rho_s$ need not be an action by $C^1$ diffeomorphisms.  It is for this reason that we work with the lifts to $\Lambda^s$.  Let $\hat{\rho}$ denote the action of $\pi_1 M$ on the leaf space $\Lambda^s$.  This is a lift of $\rho_s$ and the holonomy of the foliation  $\F^s$ on $\tilde{M}$

Let $E = (\tilde M \times \Lambda^s) / \pi_1 M$  be the suspension of $\hat\rho$.  
Fixing notation, for $p \in \widetilde{M}$, let $\ell(p) \in \Lambda^s$ denote the leaf containing $p$.  Define a section 
 $\tilde{\sigma}: \widetilde{M} \to \widetilde{M} \times \Lambda^s$  by $p \mapsto (p, \ell(p))$.  This satisfies the $\hat \rho (\pi_1 M)$-equivariance
\[ \gamma \cdot p \mapsto (\gamma \cdot p, \hat\rho(\gamma)(\ell(p)) ) \]
so induces to a section $\sigma : M \to E$.  
Since $\hat\rho$ is a $C^1$ action, the section $\tilde{\sigma}$ (and hence also $\sigma$) are $C^1$ embeddings.  Also, $\tilde \sigma$ is transverse to the leaves $\tilde M \times \{l\}$ of the horizontal foliation on $\tilde M \times \Lambda^s$ since the composition of $\tilde \sigma$ with the projection to $\Lambda^s$ is precisely the quotient map to the leaf space $\Lambda^s$, which is a non-singular $C^1$ map. 
By definition, the leaves of $\tilde \sigma(\F^s)$ are simply the intersection of $\tilde{\sigma}(M)$ with the leaves of the horizontal foliation of the suspension $E$. This means that the $C^1$ foliation $\tilde{\sigma}(\F^s)$ is transverse to $\tilde{\sigma}(\F^u)$ in $\tilde{\sigma}(\tilde{M})$.  

\begin{figure} 
 \labellist 
  \hair 2pt
   \pinlabel $\Lambda^s$ at 150 650
     \pinlabel $\H^2$ at 800 15
  \endlabellist
\begin{center}
\includegraphics[width=3in]{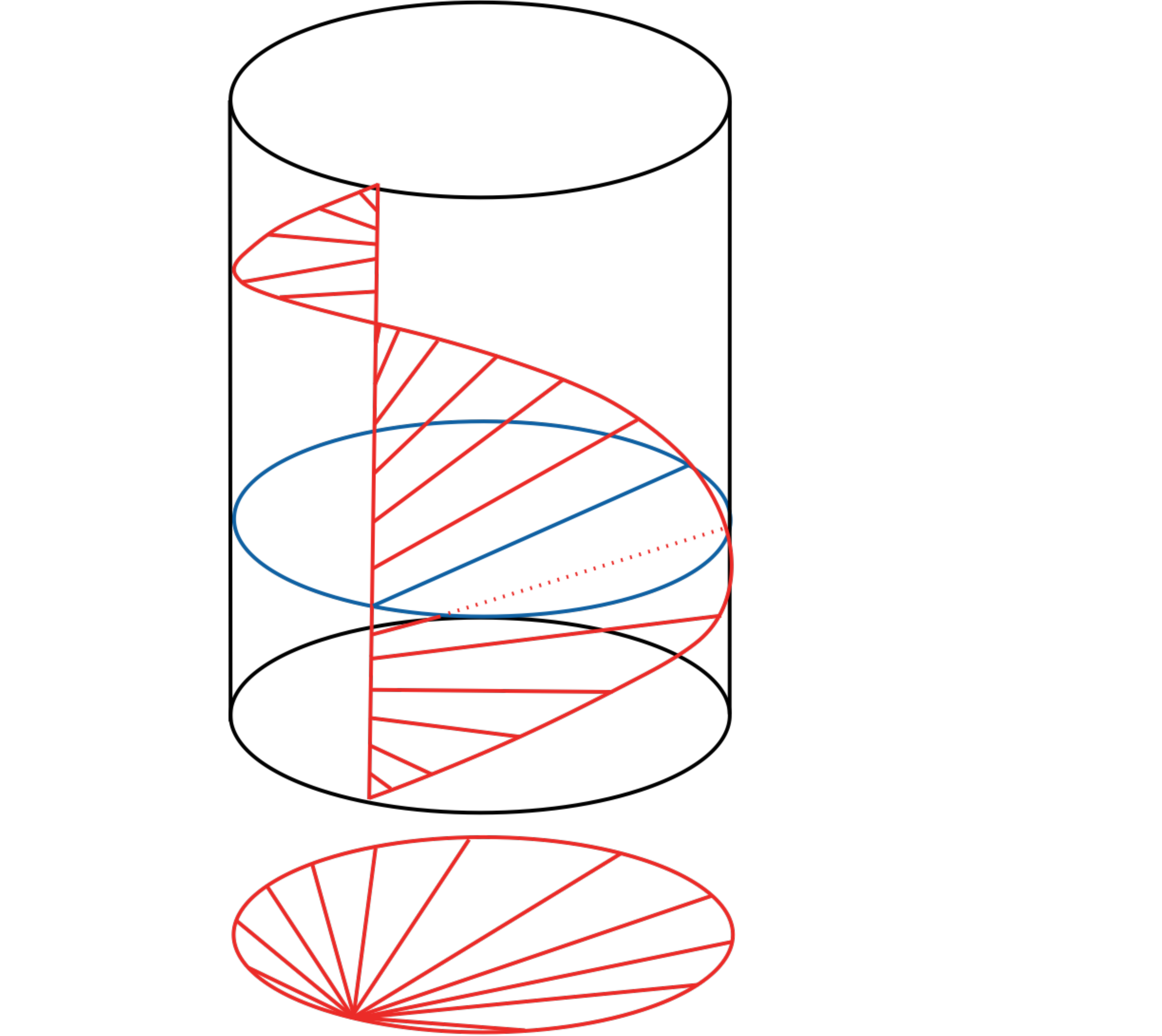}
\caption{A picture of $\tilde \sigma(\tilde{M})$ in the familiar case $M = \UT \Sigma$ as in \cite{ThurstonSlithering}.  The infinite cylinder is $\tilde \sigma (\tilde M) \cong \tilde M = \H^2 \times \Lambda^s$ i.e. the height of a point in the stack of copies of $\H^2$ corresponds to the positive endpoint of a unit tangent vector based at that point.  Horizontal planes are leaves of $\F^s$, one is shown in blue.  A leaf of $\tilde\sigma( \F^u)$ is shown in red, the hyperbolic metric on the leaf is that lifted from the projective model of $\H^2$ shown below.  In this model, one endpoint at infinity of the unstable leaf is blown up to an interval.  }\label{fig:helix}
\end{center}
\end{figure}

\paragraph{Nearby actions give nearby foliations} 
Let $\rho'$ be a small perturbation of $\rho_s$, and let $\hat \rho'$ denote the lift of $\rho'$ to $\Hom(\pi_1 M, \Homeo_\Z(\R) )$ that is a small perturbation of $\hat \rho$.  
Since $E_{\hat \rho}$ is a foliated $\R$ bundle over $M$ with $C^1$ foliation, and $\hat \rho'$ is a small perturbation of the holonomy, following the proof of Lemma \ref{lem:leafwise_imm} verbatim 
produces a homeomorphism $f_{\rho'} : E_{\hat \rho'} \to E_{\hat \rho}$ taking leaves of the suspension foliation on $E_{\hat \rho'}$ to a foliation with tangent distribution uniformly close to the horizontal distribution on $E_{\hat \rho}$.  ``Close" can be made as small as we like by choosing $\rho'$ sufficiently close to $\rho_s$.   Since $\hat \rho'$ and $\hat \rho$ are both representations into $\Homeo_\Z(\R)$, as in the last step of the proof of Lemma \ref{lem:leafwise_imm} the resulting homeomorphism will descend to a map $E_{\rho_s} \to E_{\rho'}$ of the respective fiberwise quotients.  

Let $\tilde{F}'$ denote the image of the horizontal foliation under $f_{\rho'}$.  Since  $\tilde{F}'$ has tangent distribution close to that of the horizontal distribution on $E_{\hat \rho}$, its restriction to the section $\tilde \sigma(M)$ is transverse to is transverse to $\tilde \sigma(\F^u)$.  
Abusing notation slightly, we now let $\tilde \F'$ denote the restriction of this foliation to $\tilde \sigma(M)$.  We will study the foliations $\tilde{\sigma}(\F^s)$ and $\tilde{\F'}$ on $\tilde{\sigma}(\tilde{M})$, and the induced foliations $\sigma(\F^s)$ and $\F'$ on the quotient $\sigma(M)$.

\paragraph{Leafwise quasi-geodesic foliations and the endpoint map} We adapt the line of argument carried out in Section \ref{sec:endpoint_maps}, using ``endpoint maps" to define a weak conjugacy between $\rho_s$ and $\rho'$.  
By Proposition \ref{prop:Anosov_properties} (ii), we may fix a metric on $M$ so that leaves of $\F^s$ are uniformly bi-Lipschitz equivalent to $\H^2$ and 
and the foliation $\F^u \cap \F^s$ is uniformly quasi-geodesic.  

Since the tangent distribution of $\tilde \F'$ is $C^0$ close to the horizontal in $\tilde{M} \times \Lambda^s$, for any fixed leaf $L$ of $\tilde\sigma(\F^u)$, the intersection of $\tilde \F'$ with $L$ will give a foliation of $L$ with tangent distribution close to that of the flowlines $\F^u \cap \F^s$.   Thus, under our bi-Lipschitz identification of $L$ with $\H^2$, the foliation $\tilde \F'$ is quasi-geodesic on $L$.  Recall also from Proposition \ref{prop:Anosov_properties} (ii) that flowlines of $\Phi_t$ on $L$ share a common negative endpoint, say $\xi \in \partial \H^2$.  Since the bi-Lipschitz equivalence $L \cong \H^2$ maps strong unstable leaves to horocycles and flowlines to geodesics, the leaves of $\tilde \F' \cap L$ will be nearly orthogonal to horocycles based at $\xi$.  The argument from Lemma \ref{lem:qg_line} (repeated essentially verbatim) shows that $\tilde \F' \cap L$ is uniformly quasi-geodesic.  The fact that the action of $\pi_1(M)$ on triples of ordered distinct points in $S^1_u$ is cocompact (cf.\ \cite[Proposition 7.4]{ThurstonSlithering}), together with the argument from Lemma \ref{lem:qg_line} also shows that the intersection of any leaf $L$ of $\F^u$ with $\mathcal{F}'$ is connected. 
Thus, varying $L$ we obtain a uniform quasi-geodesic foliation of $\tilde{M}$.  Denote this foliation by $\Fqg$.  Furthermore, for each leaf $L$ of $\F^u$, we have endpoint maps $e^+_L$ and $e^-_L$ taking leaves of $\Fqg$ in $L$ to their positive and negative endpoints on the ideal boundary of $L$.

By Lemma \ref{lem:uniform_foliation}, the boundaries of leaves of $\F^u$ may be identified to give a universal circle $S^1_u$, allowing us to piece together the maps $e^+_L$ and $e^-_L$ to obtain globally defined maps $e^+$ and $e^-$ from the leaf space of $\Fqg$ to $S^1_u$.  The proof of Lemma \ref{lem:cont_1} shows that, for each leaf $L$ of $\F^u$, the map $e^+_L$ is continuous.    Since $\F^u$ is a uniform foliation \cite{ThurstonSlithering}, the inclusion of any leaf into a neighbourhood $N_\varepsilon(L)$ is a uniform quasi-isometry and $L' \subset N_\varepsilon(L) $ for any sufficiently close leaf $L'$. Now nearby leaves of $\Fqg$ lying in $L,L'$ respectively of remain close on long segments, hence their end points are close in the ideal boundary $\partial_\infty N_\varepsilon(L)$ which is then canonically identified with $\partial_\infty L, \partial_\infty L' $ via the inclusion. We can now argue exactly as in Lemma \ref{lem:cont_1}, to show that the globally defined maps $e^+$ and $e^-$ are continuous on the whose manifold as well.  

\paragraph{Straightening quasi-geodesics to produce the semi-conjugacy}
To conclude the proof, we follow a modified version of the argument from  \ref{prop:positive_const}.  
Since leaves of $\F^{QG}$ are transverse to the image of the strong unstable foliation $\tilde\sigma ({\F}^{uu})$, they can be continuously homotoped via a $\pi_ 1 M$-equivariant homotopy $\tilde{h}_t$ along the one-dimensional leaves of the strong unstable foliation in such a way so that the image of each leaf of $\Fqg$ under $\widetilde{h}_1$ is the flow line of $\Phi_t$ with the same ideal endpoints in the given leaf.  
\begin{figure}[h]
\begin{center}
\includegraphics[width=3in]{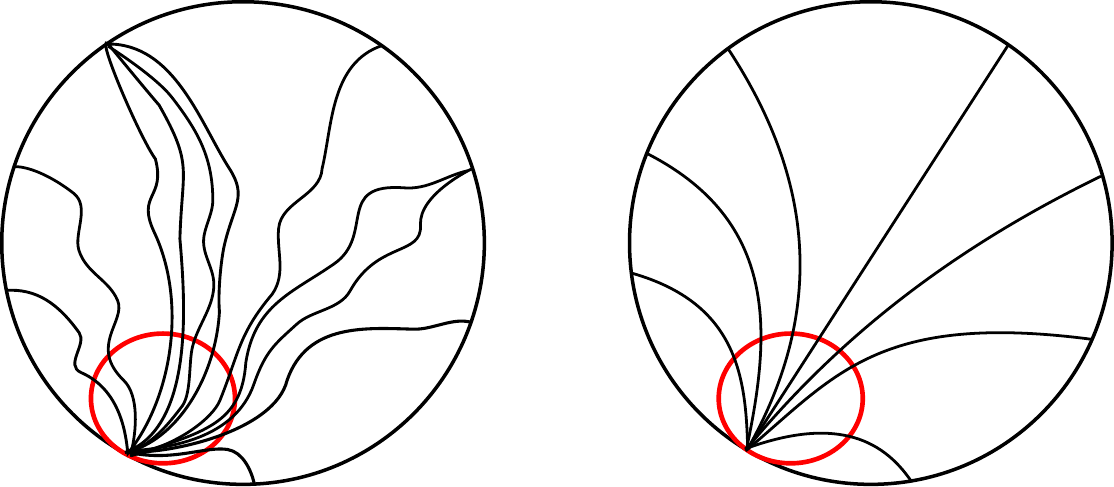}
\end{center}
\caption{Quasi-geodesics in a leaf of $\F^u$ and their images under $\tilde{h}_1$; leaves with common endpoints are identified. The red circle is a leaf of $\F^{uu}$}\label{fig:straightening}
\end{figure}

The time-one map $\tilde{h}_1$ of the homotopy descends to a map $h_1$ from the leaf space of $\F^{QG}$ to the orbit space of the flow (which we denote by $\mathcal{O}$), making the following diagram commute.
\[
  \begin{tikzcd}
    \tilde \sigma(\widetilde{M}) \arrow{d} \arrow{r}{\widetilde{h}_1} &  \widetilde{M} \arrow{d} \\
    \tilde \sigma(\widetilde{M})/\F^{QG} \arrow{r}{h_1} & \mathcal{O}.
  \end{tikzcd}
\]
We claim that for each leaf $L'$ of $\widetilde{\F}'$, its image $h_1(L')$ agrees with the image in $\mathcal{O}$ of some leaf of $\F^s$.  Equivalently, we need to show that the positive endpoint map is constant on each leaf $L'$ of $\tilde{\F}'$.  
To show this, we will use the picture given by Thurston's universal circle perspective, as stated in Proposition \ref{prop:universal_circle}.  Following this, the negative and positive endpoint maps give local (first and second) coordinates on $\mathcal{O}$.  
Fix any leaf $L'$ of $\widetilde{\F}'$.  Note first that $e^-(h_1(L'))$ is nonconstant, i.e. its image in $\mathcal{O}$ does not correspond to a vertical segment in $\mathcal{O}$.  This is simply because $L'$ intersects at least two distinct leaves of $\F^u$. We wish now to show that $h_1(L')$ is horizontal.

Suppose for contradiction that this is not the case.  
By Proposition \ref{prop:skew_Anosov_minimal}, the skew-Anosov flow $\Phi_t$ is transitive, and so its periodic points are dense.  It follows that the image of $L'$ intersects both the unstable leaf {\em and} the stable leaf of some periodic orbit.  Let $\gamma \in \pi_1M$ be the element represented by this periodic orbit, thought of as a closed curve in $M$.

\begin{figure}[h]
    \centerline{ \mbox{
\includegraphics[width=2.5in]{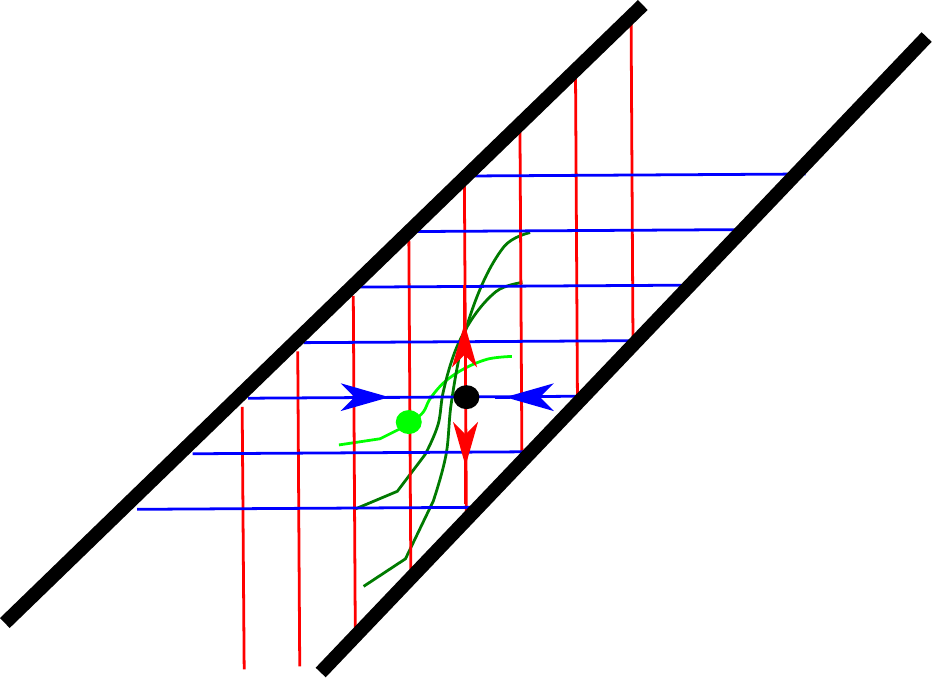}
}}
\caption{The image of a leaf in the flow space $\mathcal{O}$ which is non-horizontal is shown in light green. The black point is a periodic orbit corresponding to $\gamma \in \pi_1 M$  and the lines in darker shades of green show the first few iterates under the action of $\gamma$.}\label{Skew_Flow_not_hor}
\end{figure}

Since $\tilde{h}_1(\F^{QG})$ is a $\pi_1M$-equivariant foliation, $\gamma^n \widetilde{h}_1(L')$ are also leaves in the image of $\tilde{h}_1$. 
See Figure \ref{Skew_Flow_not_hor} for a schematic picture.
Since we are assuming that $\widetilde{h}_1(L')$ is not horizontal, then the sequence of leaves $\gamma^n \widetilde{h}_1(L')$ approaches (uniformly on compact sets) a vertical segment.  The fact that the foliation $\widetilde{\F}'$ is lifted from a flat $S^1$-bundle structure gives us compactness of the leaf space mod $\tau$, so after passing to a subsequence, in the quotient by $\tau$ the leaves $\gamma^n \widetilde{h}_1(L')$ converge to some limit leaf $L_{\infty}$.  By continuity of the map $\widetilde{h}_1$, the image of this leaf is vertical,  contradicting our earlier observation. 

We conclude that leaves of $\tilde \F'$ map to (subsets of) leaves of $\F^s$.   We now argue that leaves are sent { \em onto} leaves; in other words, the straightening map $h_1$ defines a map from the leaf space of $\tilde \F'$ to that of $\F^s$ on $\tilde{M}$.  To see this, consider first a leaf $L'$ of $\F'$ whose image contains a periodic orbit representing some $\alpha \in \pi_1 M$.   The $\pi_1 M$-equivariance of our construction means that $L'$ is also invariant under the action of $\alpha$, and its image under $h_1$ is a $\alpha$-invariant subset of a stable leaf.  We additionally know that, under the negative endpoint map, this subset contains an interval.  Thus, it must necessarily be the full leaf.  
The general case (for leaves not necessarily containing a periodic orbit) now follows from the density of periodic orbits of the flow and continuity.    

In summary, we have the following induced 
$\pi_1M$-equivariant maps, where the vertical maps 
denote the maps to the respective leaf spaces:  
\[
  \begin{tikzcd}
  \tilde \sigma (\widetilde{M})  \arrow{d} \arrow{r}{\widetilde{h}_1} &  \widetilde{M} \arrow{d} \\
\tilde \sigma(\widetilde{M})/\widetilde{\F}'   \arrow{r}{h_1} 
& \widetilde{M}/\F^{s}  = \Lambda^s 
  \end{tikzcd}
\]
Note that the map $h_1$ preserves the (weak) order of leaves in the leaf space, so it is monotone, and it is surjective and equivariant with respect to the action of  $\tau$. Thus, after quotienting out by the action of $\tau$ on $\widetilde{M}$ and $\R$ respectively, we obtain a surjective monotone map $h$ on the circle $S^1 = \R / \tau $ that provides the desired weak conjugacy.
\end{proof}

\begin{proof}[Proof of Proposition \ref{prop:slith_rigid_global}]
Let $\rho$ be a representation as in the statement of Proposition \ref{prop:slith_rigid_local}.  We show that the property of being weakly conjugate to $\rho$ is both an open and closed condition in $\Hom(\pi_1M, \Homeo_+(S^1))$.  

\paragraph{Closedness} 
It follows from work of Ghys \cite{Ghys87} and Matsumoto \cite{Matsumoto86} that, for any discrete group $\Gamma$, the closure of a conjugacy class in $\Hom(\Gamma, \Homeo_+(S^1))$ is a weak conjugacy class (called semi-conjugacy rather than weak conjugacy by these authors).  This is because weak conjugacy classes can essentially be specified by rotation numbers of elements, rotation number being a continuous function on $\Homeo_+(S^1))$, or by the integer bounded Euler class.   A detailed exposition is given in \cite[\S 2]{MannWolff}.

\paragraph{Openness} 
This follows from Proposition \ref{prop:slith_rigid_local} and approximability of weak conjugacies on $S^1$ by homeomorphisms.  (See the remark after Definition \ref{def:semiconj}.)  Let $U$ be the neighborhood given by Proposition \ref{prop:slith_rigid_local}.  
Let $\rho'$ be weakly conjugate to $\rho$ via a degree one monotone map $h: S^1 \to S^1$ satisfying $h \circ \rho' = \rho \circ h$; since $\rho$ is minimal by Proposition \ref{prop:skew_Anosov_minimal}, $h$ is continuous.
If $h_0 \in \Homeo_+(S^1)$ is a sufficiently close $C^0$ approximation to $h$, then $h_0 \rho' h_0^{-1} \in U$ so admits a neighborhood $V$ consisting of weakly conjugate representations; and $h_0^{-1} V h_0$ is the desired neighborhood of $\rho'$. 
\end{proof}

\subsection{Application: global rigidity of geometric representations} \label{sec:rig_geom}
As a first application of Theorem \ref{thm:slith_rigid}, we give a new proof of the main result of \cite{Mann}.
A second application is discussed in the next section.  Both use the following standard construction; further discussion of which can be found in \cite{Mann}.    

\subsection{Fiberwise covers of the geodesic flow} \ref{subsec:covers}
Let $\Sigma$ be a closed hyperbolic surface.  Then $\Sigma = \H^2/\rho_0(\pi_1\Sigma)$ where $\rho_0$ is an embedding as a cocompact Fuchsian sugroup of $\PSL_2(\R)$.  The action of $\rho_0(\pi_1\Sigma) \subset \PSL_2(\R)$ on $\partial_\infty \H^2 = S^1$ by M\"obius transformations gives a realization of the boundary action of $\pi_1\Sigma$.  
As in Example \ref{ex:skew_Anosov}, the corresponding suspension foliation of this representation can be naturally identified with the weak stable foliation of the geodesic flow.   

Lifts of $\rho_0$ to the extension $\Z/k\Z \to\PSL^{(k)}_2(\R) \to \PSL_2(\R)$ are precisely the holonomy representations of the weak stable foliations of the possible lifts of the geodesic flow to a $k$-fold fiberwise cover of $M \to \UT\Sigma$. 
Such lifts exist if and only if $k$ divides the Euler characteristic $\chi (\Sigma)$; in which case for a genus $g$ surface there are $k^{2g} = |\Hom(\pi_1\Sigma, \Z/k\Z)|$ distinct lifts.   These lifts can be also be distinguished dynamically:  thinking of $\PSL^{(k)}_2(\R) \subset \Homeo_+(S^1)$, via the natural identification of lifts of M\"obius transformations to the $k$-fold cover of $S^1$, one can classify distinct lifts by the rotation numbers of a standard generating set.  The images of a standard generator under different lifts differ by rigid rotations 
through angles that are multiples of $2\pi/k$.

\begin{remark} \label{rem:rotation_number}
Topologically speaking, the effect of modifying the action of a generator $\alpha$ by a rotation is to modify the closed orbits of the lifted flow that project to $\alpha$ under the map $\pi_1(M) \to \pi_1(\Sigma)$.     In detail, that some standard generator $\gamma$ for $\pi_1 \Sigma$ has image $\hat \rho(\gamma) \in \Homeo_+(S^1)$ with rotation number $2 \pi n/k$, means precisely that in the suspension of $\hat \rho$, the projection to the $S^1$ fiber of the horizontal lift of $\gamma^n$ to a closed orbit, considered as a map $S^1 \to S^1$, has degree $k$.   We will use this perspective again in the proof of Theorem \ref{thm:InequivalentAnosov}.  
\end{remark} 

The following is the main result of \cite{Mann} (reproved using a different argument by Matsumoto in \cite{Matsumoto16}).  
By quoting Theorem \ref{thm:slith_rigid}, we may give another, shorter independent proof.  

\begin{theorem}[Mann \cite{Mann}]\label{cor:Mann}
Let $\Sigma$ be a surface of genus $g\geq 2$, and $\rho : \pi_1\Sigma \rightarrow \PSL_2(\R) \subset \Hom(\pi_1\Sigma, \Homeo_+(S^1))$ an embedding as a cocompact Fuchsian group. Consider any lift $\hat{\rho}$ of this action to the $k$-fold cover of $S^1 \stackrel{z^k} \longrightarrow S^1$:
\[
  \begin{tikzcd}
   & \Homeo_+^{(k)}(S^1) \arrow{d}{z^k} \\
    \pi_1\Sigma\arrow{ur}{\hat{\rho}}  \arrow{r}{\rho} & \Homeo_+(S^1). 
  \end{tikzcd}
\]
Then the connected component of $\hat{\rho}$  in $\Hom(\pi_1\Sigma, \Homeo_+(S^1))$ is a single weak conjugacy class.
\end{theorem}

\begin{proof}
Equip $\Sigma$ with a hyperbolic metric and let $\UT\Sigma$ be its unit tangent bundle.  The geodesic flow on $\UT\Sigma$ is skew-Anosov, so determines a slithering action $\rho_s: \pi_1(\UT\Sigma) \to \Homeo_+(S^1)$.  We consider lifts $\hat{\rho}$ of $\rho = \rho_s$ as per our discussion above.   Each lift $\hat{\rho}$ is the holonomy of the lift of the weak stable foliation of the geodesic flow to a $k$-fold fiberwise cover of $\UT\Sigma$.  Let $M$ be such a $k$-fold cover, so $\pi_1M$ sits in a central extension 
\[
1 \longrightarrow \Z = \langle z \rangle \longrightarrow \pi_1M \longrightarrow \pi_1\Sigma \longrightarrow 1.
\]
The lift of geodesic flow to $M$ is also skew-Anosov, so has a slithering $\rho_{s_k}: \pi_1M \to \Homeo_+(S^1)$.   It is easily verified from the definitions that these representations satisfy $\rho_{s_k}(z) = \id$, so descend to representations $\pi_1\Sigma \to \Homeo_+(S^1)$, which are precisely those appearing in the statement of Theorem \ref{cor:Mann}.   Theorem \ref{thm:slith_rigid} states that the representation $\rho_{s_k}$ is globally rigid in $\Hom(\pi_1M, \Homeo_+(S^1))$.   This now implies rigidity of the surface group action obtained by restricting $\rho_{s_k}$ to $\pi_1(\Sigma)$, since any element of its connected component in $\Hom(\pi_1\Sigma, \Homeo_+(S^1))$ can be extended to a representation of $\pi_1(M)$ by declaring the central $\Z$ subgroup to act trivially.  
\end{proof}

A further consequence of Theorem \ref{thm:slith_rigid} is the following.  

\begin{corollary}
Let $M$ be a closed $3$-manifold admitting a skew-Anosov flow. Then the component of the space $\Hom(\pi_1M,\textrm{Homeo}_+(S^1))$ with trivial Euler class is not connected. 
\end{corollary}
\begin{proof}
The slithering action $\rho_s$ corresponds to Thurston's {\em universal circle action} given by compactifying leaves of $\F^s$ in the universal cover, as described above.  
The Euler class of this action agrees with the Euler class of $\F^s$, which is trivial since the tangent bundle to $\F^s$ admits a nowhere vanishing section determined by the flow. But $\rho_s$ is not in the same component as the trivial representation by Proposition \ref{prop:slith_rigid_global}, although they have the same Euler class.
\end{proof}

\section{Topologically inequivalent Anosov flows on hyperbolic manifolds} \label{sec:InequivalentAnosov}
In this section we prove Theorem \ref{thm:InequivalentAnosov} using ideas developed above.   Recall that two non-singular flows on a manifold are {\em topologically equivalent} if the one-dimensional foliations given by their flow lines are conjugate as foliations. 
We will consider examples of skew-Anosov flows obtained by lifting geodesic flows to a $k$-fold fiberwise cover of the unit tangent bundle of a hyperbolic surface (for large $k$), then performing integral Dehn surgery along a closed orbit.  (We assume that the reader has some familiarity with hyperbolic Dehn surgery e.g.\ as described in \cite{Thurston_notes}. A brief description of surgery for flows is given below in paragraph \ref{sec:Dehn}.) 

\begin{remark}[Lifts to a fiberwise covers up to topological equivalence, \cite{BaFe}] \label{rem:equiv_lifts}
Recall from Section \ref{subsec:covers} that the different lifts of geodesic flow on $\UT\Sigma$ to the $k$-fold fiberwise cover of $\UT\Sigma$ are determined by cohomology classes in $H^1(\UT\Sigma,\Z_k)$ that pair with the generator to give 1.   It is straightforward to check that the group of fiberwise rotations (i.e. smooth gauge transformations of the cover), which can be identified with the group of smooth maps $\Sigma \to U(1) = S^1$, acts transitively on this affine subspace of cohomology. Note that these equivalences change the {\em isotopy class} of the flow as the homotopy classes of periodic orbits will change. However, the mapping class group of $\UT\Sigma$ also contains the mapping class group of the base. By carefully considering the action on both the flow and the covering Barbot and Fenley \cite{BaFe} show that there are two distinct Anosov flows up to equivalence in the case that $k$ is even and only one in the case that $k$ is odd. Moreover, in this case, all equivalences can be realized by diffeomorphisms of the ambient manifold.
\end{remark}

Barbot and Fenley's result as described above implies that some extra ingredient is required to produce many inequivalent flows.  This is where Dehn surgery comes into the picture.  

\subsection*{Asymmetric Knots}  
The first ingredient is the following construction of highly asymmetric filling geodesics on surfaces.  We will later lift these to a fiberwise cover of the unit tangent bundle and perform Dehn surgery along the resulting curve.    

\begin{notation} 
Let $T$ be a one-holed torus, and $a, b$ simple, oriented curves representing standard generators of $\pi_1(T)$. 
 Let $c_{m,n}$ denote a properly embedded arc with endpoints on $\partial T$, constructed by first following a simple subarc that wraps $n$ times in the $a$ direction, and then a simple arc wrapping $m$ times in the $b$ direction, as shown in Figure \ref{fig:torus}.   The complimentary regions to $c_{m,n}$ are quadrilaterals, with the exception of one 5-gon, and one 7-gon containing an arc of $\partial T$.
\end{notation} 

\begin{figure} 
  \labellist 
  \hair 2pt
   \pinlabel $b$ at 0 150
     \pinlabel $a$ at 200 5
   \endlabellist
\begin{center}
\includegraphics[width=1.6in]{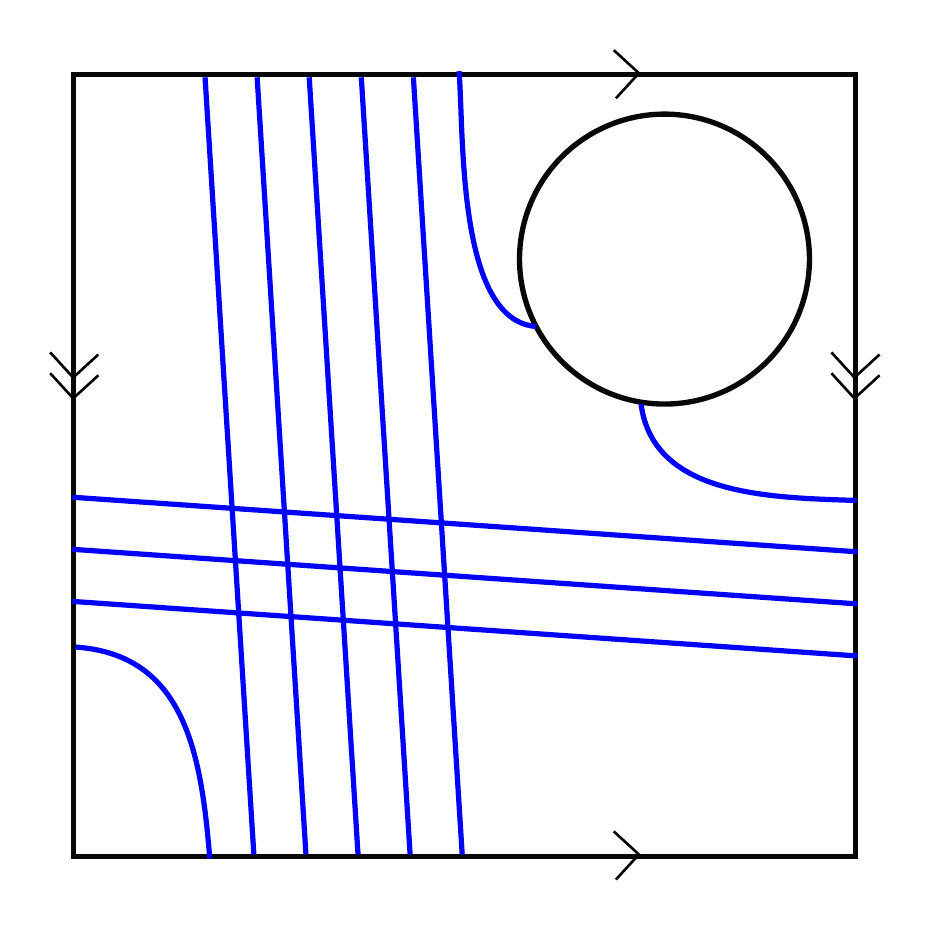}
\caption{The curve $c_{3,5}$ on $T$}\label{fig:torus}
\end{center}
\end{figure}

The next lemma says that arcs of the form $c_{m,n}$ on disjoint one-holed tori can be pieced together to give a curve on a higher genus surface that has the same self-intersection pattern as its geodesic representative in any hyperbolic metric.  For the set up, 
fix $g > 2$ and fix a decomposition of $\Sigma_g$ into $g$ pucntured tori $T_1, T_2, \ldots T_g$ and one $g$-holed sphere $S$.   Let $m_1, n_1 = 3, 5$, and for $i = 2, 3, \ldots g$, choose $m_i > n_i+1$ and $n_i > m_{i-1} + 1$.  Let $c$ be a closed curve on $\Sigma_g$ whose restriction to $T_i$ is the arc $c_{m_i, n_i}$, and such that $c$ has no points of self-intersection in $S$.   

\begin{lemma}\label{lem:geod_katie} 
For any curve $c$ as above, the following hold:
\begin{enumerate} 
\item For any hyperbolic metric on $\Sigma_g$, the geodesic representative of $c$ has the same self-intersection pattern as $c$.  
\item If $f$ is a finite order homeomorphism of $\Sigma_g$ such that $f(c)$ is ambiently isotopic to $c$, then $f$ is the identity.  The same statement holds when $c$ is replaced with its geodesic representative in any hyperbolic metric.   
\end{enumerate} 
\end{lemma} 

\begin{figure} 
\begin{center}
\includegraphics[width=3in]{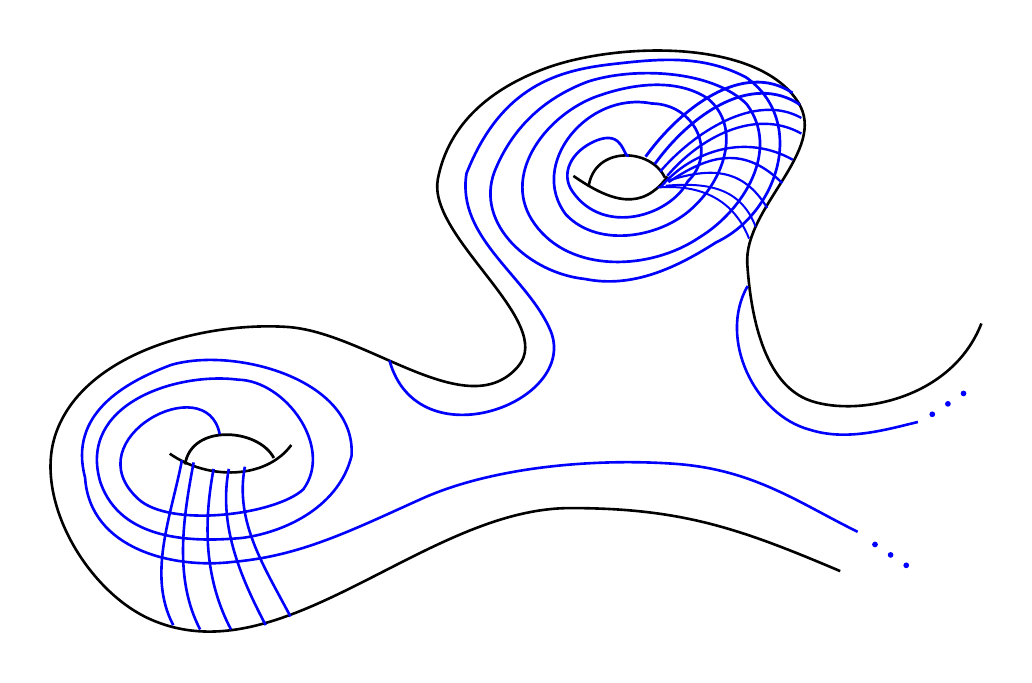}
\caption{Part of the curve $c$ from Lemma \ref{lem:geod_katie}}\label{fig:curve}
\end{center}
\end{figure}

\begin{proof}
First, note that by construction no complimentary region of $c$ is a monogon, bigon, or triangle.  
Fix any hyperbolic metric on $\Sigma_g$, and consider the geodesic representative $c_{\text{geo}}$ of $c$ in this metric.   We claim that it has the same combinatorics as the curve $c$ depicted in the figure.  To see this, we use the {\em disk flow} of Hass and Scott defined in \cite{HassScott}.  Starting with a curve in general position, this flow may change the combinatorics of a curve via $a)$ eliminating a monogon or bigon bounding a disk (thus decreasing the self-intersection number of the curve) or $b)$ moving one edge of a triangle across the opposite vertex, preserving the self-intersection number.   Hass and Scott show \cite[Thm 2.1, 2.2]{HassScott} that any curve is homotopic to a representative with minimal self-intersection number through this process, and that any two distinct representatives of a curve, each having minimal self-intersection number, are homotopic to each other through moves of type $b)$ and ambient isotopy of the surface.  
Since no complimentary regions of $c$ are monogons or bigons, moves of type $a)$ are not possible.  Thus, $c$ has minimal self-intersection number.  Since no regions are triangles, Hass and Scott's theorem implies that $c_{\text{geo}}$ (which is a geodesic, hence also has minimal self-intersection number) must be attainable from $c$ by ambient isotopy of $\Sigma_g$.  This proves (1).  

For the second assertion, suppose that $f$ is a finite order homeomorphism of $\Sigma_g$, and $h$ a homeomorphism isotopic to identity such that $hf(c) = c$, setwise.  Then $hf$ induces an automorphism of the graph on $\Sigma_g$ formed by the image of $c$.   We claim that this graph has no nontrivial automorphisms.   To see this, note that (with appropriate choice of orientations on the tori $T_i$), the two complimentary regions to $c$ that intersect $S$ are a $p$-gon and $q$-gon for $p \neq q > 5$, so much each be preserved.  The choice of $n_i > m_{i-1}+1$ and $m_i > n_i+1$, ensures that the ``grids" of quadrilaterals which distinguish the torus subsurfaces have no nontrivial symmetries and cannot be permuted, from which one deduces inductively that each complimentary region is fixed.  We leave the details as an elementary exercise.  
Thus, $hf$ induces a trivial graph automorphism.   Since $hf$ is isotopic to the finite order homeomorphism $f$, and preserves each complimentary region of the filling curve $c$, it must be isotopic to the identity,  as follows from the standard Alexander trick argument, hence $f$ is the identity.
This argument only relied on the combinatorics of $c$ and its complimentary region, so by the first assertion it also holds when $c$ is replaced by a geodesic representative. 
\end{proof}

The next step is to show that removing the geodesic representative of such a curve $c$ from the 3-manifold $\UT\Sigma$, or removing a lift of it to a fiberwise cover of $\UT\Sigma$, gives a  $3$-manifold which admits a complete hyperbolic structure.  This comes from the following folkloric result.  

\begin{lemma}[Calegari/Folklore]\label{lem:Calegari_Folklore}
Let $c \subset \Sigma$ be a closed, filling geodesic in a hyperbolic surface.  Then the complement of its image in $\UT\Sigma$ is irreducible and atoroidal.  More generally, if $M \to \UT\Sigma$ is a $k$-fold fiberwise cover, and $K$ a connected component the preimage of $c$ in $M$, then $M-K$ is irreducible and atoroidal.
\end{lemma}
This is stated and proved in detail for the case where $M= \UT\Sigma$ in \cite[Appendix B]{Foulon_Hasselblatt}, where Foulon and Hasselblatt use it to construct examples of contact Anosov flows on hyperbolic manifolds. 
However, the proof carries over verbatim when the bundle $\UT\Sigma \to \Sigma$ is replaced by any finite fiberwise cover.  See also \cite{DannyBlog} for an alternative exposition.

For the next two lemmas we will use the following set-up. 
Let $M \to \UT\Sigma$ be a $k$-fold fiberwise cover of $\UT\Sigma$ where $\Sigma$ is a hyperbolic surface of genus $g \geq 3$.  As in Lemma \ref{lem:Calegari_Folklore} above, let $K \subset M$ be a connected component of the preimage of a geodesic $c$ in $\Sigma$, where $c$ is chosen as in Lemma \ref{lem:geod_katie}.  

\begin{lemma} \label{lem:new_slope1}
If $h \in \Homeo(M)$ is homotopic to a finite order homeomorphism of $M$, and $h(K)$ is isotopic to $K$, then $h$ is homotopic to the identity. 
\end{lemma} 
Note that $h$ need not be equal to the identity, for instance it may rotate the fibers of the fibration $M \to \Sigma$.  
\begin{proof}[Proof of Lemma \ref{lem:new_slope1}]
Consider the action of $h$ on $\pi_1M$.  This action preserves the center of $\pi_1M$, which is the fundamental group of the fiber, so descends to an action $\overline{h}$ on $\pi_1\Sigma$ modulo inner automorphisms, i.e. a map $\overline{h} \in \textrm{Out}(\pi_1\Sigma)$.  Since $h$ is finite order,  Nielsen realization implies that $\overline{h}$ can be realized as an isometry for some hyperbolic structure on $\Sigma$.  Since $h$ preserves $K$ up to isotopy, the isometry realizing $\overline{h}$ preserves $c$ up to free homotopy, so preserves the geodesic representative of $c$ in this metric.  Since $c$ was chosen as in Lemma \ref{lem:geod_katie}, this isometry is in fact trivial, so $\overline{h}$ is a trivial outer automorphism of $\pi_1\Sigma$.

Now $M$ is a $K(\pi, 1)$ space, so homotopy classes of maps $M \to M$ are determined by the action on the fundamental group, and so homotopy classes of maps that induce the trivial outer automorphism of $\pi_1\Sigma$ can be identified with elements of $\Hom(\pi_1\Sigma, \Z)$, which is torsion free.  Since $h$ was assumed finite order, it must therefore be homotopic to the identity.
\end{proof} 

The following is the main technical result of this section.  

\begin{lemma} \label{lem:new_slope2}
Let $M_p$ denote the integral Dehn filling on $M - K$ of slope $p$.   After excluding finitely many slopes, the following hold.  
\begin{enumerate}
\item  The manifold $M_p$ is hyperbolic and each homeomorphism of $M_p$  preserves the homotopy class of the core curve of the filling torus.
\item Each homeomorphism $\phi \in \Homeo_+(M_p)$ that preserves the core of the filling torus determines a  homeomorphism $\phi_M$ of $M$ that preserves the isotopy class of $K$, agrees with $\phi$ away from a tubular neighbourhood of $K$, and is homotopic to the identity.    This tubular neighborhood of $K$ may be chosen arbitrarily small.  
\end{enumerate}
\end{lemma} 

\begin{remark} \label{rem:new_slope2}
The first point is true (after excluding finitely many slopes) whenever $K$ is obtained by lifting a filling geodesic on the surface.  The fact that $\phi_M$ in point (2) is homotopic to the identity comes from our choice of $c$ from Lemma \ref{lem:geod_katie}.   
\end{remark} 

\begin{proof}[Proof of Lemma \ref{lem:new_slope2}] 
By Lemma \ref{lem:Calegari_Folklore} the complement  $M - K$ is atoroidal and irreducible, and thus admits a complete hyperbolic metric by geometrisation.  Fix this hyperbolic metric, and consider the action of the isometry group $\Isom(M - K)$ on the fundamental group of the cusp, which we identify with $\Z \times \Z$ using generators coming from the meridian and longitude of a tubular neighborhood of $K$.   Recall that, by Mostow--Prasad rigidity, $\Isom(M - K)$ is finite.  For each of the isometries whose action on $\Z \times \Z$ is not by $\pm I$, record any eigenspace of eigenvalue $\pm 1$.  This gives us a collection of finitely many slopes, which we exclude from the possible slopes of Dehn filling.  

Since $M-K$ is hyperbolic, Thurston's hyperbolisation theorem \cite{Thurston_notes} states that, with finitely many exceptions, the result of Dehn filling $M-K$ is a hyperbolic manifold, in which the core of the filling torus is a closed geodesic of shortest length for this hyperbolic structure.   Excluding these finitely many exceptional slopes as well, we claim that $M_p$ will have all the desired properties.  

First, suppose that $\phi$ is a homeomorphism of $M_p$.  By Mostow rigidity, it is homotopic to an isometry; denote this isometry by $\psi$.  Since the core of the filling torus is a geodesic of shortest length, it is preserved by $\psi$, so $\phi$ preserves this curve up to homotopy, finishing the proof of item (1).

For the second item, observe that $\psi$ induces a homeomorphism of the cusped manifold $M - K$ preserving the (unoriented) isotopy class of a longitude given by the Dehn Surgery slope, which we identify with $K$.  Again, by Mostow Rigidity, this homeomorphism is homotopic to an isometry, and by our restriction on the choices of slope $p$, we conclude that the action on the fundamental group of the cusp is by $\pm I$. 
This means also that $\psi$ can be extended to a homeomorphism, say $\psi_M$, of $M$ by coning off over meridian discs.  
Furthermore, since an isometry of a complete hyperbolic manifold has finite order, we take an extension that also has finite order, since the extension over meridian discs preserves this property.

Suppose now that $\phi$ itself has the additional property that it preserves the core of the Dehn filling torus.  By the same argument as above, $\phi$ then induces a homeomorphism of $M-K$ preserving the isotopy class of the longitude $K$ and inducing $\pm I$ on the fundamental group of the cusp and so we can extend the action of $\phi$ over meridian discs to give a homeomorphism $\phi_M$ of $M$ preserving the isotopy class of $K$.   

Since $M$ is a $K(\pi, 1)$ space, the extension of any map over a tubular neighbourhood $N$ of $K$ is well-defined up to homotopy. Moreover, any homotopy of maps on $M - N$ extends to $M$.  In particular, if $\psi$ is the isometry homotopic to $\phi$, using the notation as above, then $\phi_M$ is homotopic to $\psi_M$, which is finite order.  
By Lemma \ref{lem:new_slope1}, we conclude that $\phi_M$ is homotopic to the identity.  
Finally, since $\phi$ preserves $K$, this homeomorphism $\phi_M$ can be obtained from $\phi$ by undoing the original Dehn surgery in an arbitrarily small neighborhood of $K$.   
\end{proof}  

\subsection{Dehn Surgery and Anosov flows} \label{sec:Dehn}
 Given any Anosov flow and a periodic orbit $\gamma$ Goodman \cite{Goodman} and Fried \cite{Fried} have described how to perform integral Dehn surgery on $\gamma$ in a manner compatible with the flow, giving the following. 
\begin{proposition}[Dehn surgery on Anosov flows  \cite{Fried,Goodman}]\label{prop:Dehn_surgery}
Let $\Phi_t$ be an Anosov flow on a manifold $M$ and let $\gamma$ be a periodic orbit. Then the manifold $M_p (\gamma)$ obtained by integral surgery of slope $p$ admits an Anosov flow that is conjugate to the original flow away from the core of the filling torus.
\end{proposition}

Goodman's original construction in \cite{Goodman} produces a smooth Anosov flow.  Fried \cite{Fried} gives an alternative construction which has, a priori, less regularity, but has the property that the dynamics of the flow after surgery are identical to those of the original flow in the complement of the periodic orbit given by the core of the Dehn filling torus.  In outline, one simply blows up $M$ along the normal bundle of the periodic orbit to obtain a manifold homeomorphic to the complement of a small open neighborhood of $\gamma$ in $M$, with a torus boundary to which the flow extends in a natural way, having four periodic orbits on the boundary.  Choosing a foliation of the torus boundary by circles transverse to the flow, such that each circle leaf intersects each of the periodic orbits in a single point and identifying each circle to a point, one obtains a flow on an integral Dehn-filling of $M - \gamma$ so that the core of the filling torus (the points obtained by collapsing circles) is a periodic orbit.

While the dynamics under Goodman's construction are somewhat mysterious, in Fried's version as described above it is obvious that any Dehn surgery can be undone, on the level of Anosov flows, by an inverse surgery. 
The drawback of Fried's construction is that the flows he constructs are not obviously genuinely Anosov, they are only {\em topologically} Anosov. 
It has been largely assumed in the literature that both these surgeries produce topologically equivalent flows, so that in both cases one obtains flows that are Anosov in the usual sense. This has only recently been settled by Mario Shannon \cite{Shannon} for transitive flows, which includes as a special case surgery of skew-Anosov flows (the case of interest to us).  This will be crucial in our construction. 

The surgery construction as well as some of its properties have been analyzed by Fenley \cite{Fenley}.  He shows in particular
that surgery on certain Anosov flows produces skew-Anosov examples.  We note this for future use.  
\begin{proposition}[Dehn surgery on skew-Anosov flows \cite{Fenley}]\label{prop:Dehn_surgery_Fenely}
If the original flow is a cover of the geodesic flow on $\UT \Sigma$, then for $p > 0$ the flow on $M_{-p} (\gamma)$ given by Dehn surgery of slope $-p$  is skew-Anosov. 
\end{proposition}

\subsection*{Constructing inequivalent Anosov flows}
Using the tools above, we now produce examples of hyperbolic 3-manifolds supporting $N$ topologically inequivalent Anosov flows, proving Theorem \ref{thm:InequivalentAnosov}. Recall that this will be done by performing Dehn surgery on fiberwise covers of the unit tangent bundle of a hyperbolic surface.

\begin{proof}[Proof of Theorem \ref{thm:InequivalentAnosov}]  
Let $\Sigma$ be a hyperbolic surface, with hyperbolic structure defined by a representation $\rho: \pi_1\Sigma \to \PSL_2(\R)$.  Fix some $k \in \N$ dividing the Euler characteristic of $\Sigma$, for concreteness one may take $k = g-1$, where $g$ is the genus of $\Sigma$.  
We will give a construction that produces a number of inequivalent skew-Anosov flows via surgery on the $k$-fold cover of $\UT\Sigma$, where that number grows linearly in $k$ (and hence can be taken as large as desired by taking $g$ large). 

Recall from Section \ref{subsec:covers} that for fixed $k$, the lifts of $\rho$ to the $k$-fold central extension of $\PSL_2(\R)$ are in bijective correspondence with $\Hom(\pi_1\Sigma, \Z/k\Z)$, parametrized by the rotation numbers of a standard set of generators of $\pi_1\Sigma$.   As discussed in Remark \ref{rem:rotation_number}, these lifts can also be distinguished by understanding the degree of projection to the fiber of horizontal lifts of closed curves from $\Sigma$ to the suspension $E_{\hat \rho}$.
Each lift defines an Anosov flow on the $k$-fold fiberwise cover of $\UT\Sigma$, (whose weak stable foliation is the suspension $E_{\hat \rho}$ of the lift $\hat \rho$ of $\rho$) but, as noted in Remark \ref{rem:equiv_lifts}, these are all topologically equivalent flows.   To produce inequivalent flows, we will use Dehn surgery along the natural lifts of a fixed filling geodesic as constructed in Lemma  \ref{lem:geod_katie}.  

\subsection*{Set-up and standing assumptions} 
Let $T_1 \subset \Sigma$ be a one-holed torus and $c$ a geodesic on $\Sigma$  as in Lemma \ref{lem:geod_katie}, with $\alpha_1, \beta_1$ the standard generators of $\pi_1(T)$.  Complete this to a standard generating set $\alpha_2, \beta_2, \ldots \alpha_g, \beta_g$ for $\pi_1\Sigma$.   We will consider lifts of $\rho$ that differ only on $\alpha_1$ and $\beta_1$, agreeing on all other generators.  

Identify the curve $c$ with an element of $\pi_1\Sigma$ and fix a lift $\hat{\rho}$ of $\rho$ to $\PSL^{(k)}_2(\R)$ such that $\hat{\rho}(c)$ has rotation number 0. Topologically, having rotation number zero corresponds to the fact that the ``horizontal lift''of $c$, meaning the pre-image of the geodesic $c$ under the covering map $E_{\hat{\rho}} \to \UT\Sigma$, has $k$ connected components, each one a periodic orbit of the lift of the geodesic flow to $E_{\hat{\rho}}$.  
The following argument shows that there are at least $\lfloor k/3 \rfloor$ choices for such lifts $\hat{\rho}$; all of which agree on $\alpha_i,\beta_i$ for $i \ge 2$: Recall that $c$, as an element of $\pi_1\Sigma$ has the form $w \alpha_1^3 \beta_1^5$ where $w$ is a word in $\alpha_2, \beta_2, \ldots \alpha_g, \beta_g$.   Varying the lifts of $\rho(\alpha_1)$ and $\rho(\beta_1)$, while preserving the chosen lifts of the other generators amounts to replacing $\hat\rho(\alpha_1)$ with its composition with a rotation by $\frac{2 \pi r_1}{k}$, and $\hat\rho(\beta_1)$ with its composition by some rotation of the form $\frac{2 \pi r_2}{k}$.  This changes the rotation number of $\hat\rho(c)$ by $3r_1 + 5r_2$ mod $k$.  For any choice of $r_1 \cong 5x$ mod $k$, we may choose $r_2 \cong -3x$ mod $k$ so that the rotation number does not change.  

We will further restrict our choice of lifts of $\rho$ so that the horizontal lifts of $c$ are all isotopic curves in the $k$-fold cover of $\UT\Sigma$.   Following the discussion above, when $k$ is large, we may choose $r_1$ and $r_2$ to vary by only a small family of rotations, 
 so that the holonomies of the lifted representations, which differ by rigid rotation of $\frac{2 \pi r_i}{k}$, remain $C^\infty$ close to each other in $\Hom(\pi_1\Sigma, \Diff(S^1))$.   This will give us some number $C(k)$ of lifts of $c$ which are sufficiently close to each other to be isotopic, where $C(k)$ grows linearly in $k$.  While  the genus of $\Sigma$ must increase as $k$ increases (since $k$ needs to divide $\chi(\Sigma)$ for a lift to exist), this does not pose any problem as we are performing these modifications only over the fixed torus subsurface $T_1$. 
Restricting to such lifts of $\rho$,  fix $p = p(k) \in \N$ large enough so that Proposition  \ref{prop:Dehn_surgery_Fenely} ensures that Dehn surgery of slope $-p$ 
on any connected component of any horizontal lift of $c$ for any one of this restricted class of lifts gives a skew-Anosov flow.  Further restricting $p$ if needed, we may also ensure that all homeomorphisms of the Dehn-surgered manifold preserve  the free homotopy class of the core curve by Lemma  \ref{lem:new_slope2}.

The restriction we imposed on our lifts of $\rho$ ensuring that connected components of lifts of $c$ are always isotopic means that performing a slope $-p$ Dehn surgery on any horizontal lift of $c$ to any of the covers will produce diffeomorphic hyperbolic manifolds.    Thus, the remainder of the proof is devoted to showing that the flows produced in this way are inequivalent whenever the lifts of $\rho$ differ among our $C(k)$ choices.   This means that the Dehn-surgered manifold described above admits $C(k)$ inequivalent skew-Anosov flows.   For this, we need to describe the construction a bit more carefully, setting some more precise notation along the way.   
 
\subsection*{Proof of inequivalence of flows} 
Fixing notation, let $M$ denote the $k$-fold fiberwise cover of $\UT\Sigma$ and let $\hat{\rho}$ and $\hat{\rho}'$ be two lifts of $\rho$ chosen so as to satisfy the restrictions imposed above.   The manifold $M$ is, topologically, both the suspension $E_{\hat{\rho}}$ of $\hat{\rho}$, and the suspension of $\hat{\rho}'$.  
Since $\hat{\rho}'$ is close to $\hat{\rho}$ (because of our restrictions), we may realize $\hat{\rho}'$ as the holonomy of a foliation on $M$ that is $C^1$ (and in this case actually $C^\infty$) close to the horizontal foliation defined by $E_{\hat{\rho}}$.  Going forward, we let $E_{\hat{\rho}'}$ denote $M$ equipped with this nearby foliation.  

Fix a connected component $K$ of the horizontal lift of $c$ to $E_{\hat{\rho}}$, and a connected component $K'$ of the horizontal lift of $c$ to $E_{\hat{\rho}'}$, isotopic to $K$ in $M$.   It will be convenient to fix an identification of $K$ and $K'$, so let $g: M \to M$ be an isotopically trivial homeomorphism such that $g(K) = K'$.  Then $g\Phi_t'g^{-1}$ and $\Phi_t$ each have $K$ as a periodic orbit.   Now perform integral Dehn--Fried--Goodman surgery of slope $-p$ on the knot $K$ to modify the flow $\Phi_t$ to a new skew-Anosov flow $\Psi_t$ on the Dehn-surgered manifold $M_{-p}$, and separately perform integral Dehn--Fried--Goodman surgery of slope $-p$ on $K$ to modify $g \Phi' g^{-1}$ to obtain a flow $\Psi'_t$ on $M_{-p}$.  Note that the latter construction is simply the result of performing surgery on the knot $K'$ in $M$, under our identification of $K$ and $K'$ via $g$.

What we wish to show is that $\Psi_t$ and $\Psi_t'$ are inequivalent.  Suppose for contradiction that this is not the case, so there is some homeomorphism $f:M_{-p} \to M_{-p}$ taking flowlines of $\Psi_t$ to flowlines of $\Psi_t'$.   Let $\gamma$ denote the core of the filling torus on $M_{-p}$.    By Lemma  \ref{lem:new_slope2}, $f(\gamma)$ and $\gamma$ lie in the same free homotopy class, so by Proposition \ref{prop:freely_homotopic}, there is a homeomorphism $h$ of $M_{-p}$, inducing some power of $\eta$ on the flow space of $\Psi'_t$, such that $hf(\gamma) = \gamma$.   So we now may as well consider $hf$ as the homeomorphism conjugating the foliations defined by the two flows.  

Restricting $hf$ to $M _{-p} - \gamma$ defines a homeomorphism $\phi$ of $M -K$.  As in Lemma \ref{lem:new_slope2}, this determines a homeomorphism $\phi_M$ of $M$ agreeing with $\phi$ on the complement of a neighborhood of the end, a neighborhood which can be chosen as small as we wish.  Choose such a neighborhood small enough so as {\em not} to contain any horizontal lift of any power of the curves $\alpha_1$ or $\beta_1$ to either $E_{\hat{\rho}}$ or to (the conjugate by $g$ of) $E_{\hat{\rho}'}$.   By Lemma \ref{lem:new_slope2} (2), the map $\phi_M$ is homotopic to the identity, and, by construction, outside of a small neighborhood of $K$, $\phi_M$ maps flowlines of $\Phi_t$ to those of $g\Phi_t'g^{-1}$.  
This is where we will derive a contradiction.  The curves $\alpha_1$ and $\beta_1$ each have a power, say $n$ and $m$ which admits a horizontal lift to a closed orbit of the flow $\Phi_t$.  Let $\hat{\alpha}_1$ and $\hat{\beta}_1$ denote these closed orbits.   Since $\phi_M$ is homotopic to the identity, it maps these to closed orbits of $g\Phi_t'g^{-1}$ that are freely homotopic to the orbits $\hat{\alpha}_1$ and $\hat{\beta}_1$, respectively.  In particular, the projection of 
$\phi_M(\hat{\alpha}_1)$ and $\phi_M(\hat{\beta}_1)$ to curves on $\Sigma$ are freely homotopic to $\alpha_1^n$ and $\beta_1^m$.   Thus, 
$\phi_M(\hat{\alpha}_1)$ and $\phi_M(\hat{\beta}_1)$ are also the (conjugates under $g$ of) horizontal lifts of $\alpha_1^n$ and $\beta_1^m$ to closed orbits of $\Phi_t'$.   By design, we chose $\hat{\rho}'$ to give these curves different rotation numbers than their rotation numbers under $\hat{\rho}$.   By the discussion in Remark \ref{rem:rotation_number}, this means that their horizontal lifts are not freely homotopic, since they wind a different number of times around the fibers over the representative curves on $\Sigma$. This gives the desired contradiction.
\end{proof}


\section{Further questions}

We conclude by suggesting a few directions for further study.  
\subsection{On boundary actions and rigidity} 

\begin{problem} \label{prob:stable_group} 
Suppose that $\Gamma$ is a hyperbolic group with Gromov boundary a topological sphere.  Is the action of $\Gamma$ on its boundary topologically stable?  
\end{problem} 

As a starting point to this problem, one could look for a new proof, or a direct modification of our proof, of Theorem \ref{thm:local_rig} that leans more heavily on coarse geometry (quasi-geoesdics) and less heavily on the Riemannian structure of $M$ (stable foliations for geodesic flow).  While Problem \ref{prob:stable_group} is intended to be strictly more general, the issue of which hyperbolic groups with sphere boundary are {\em not} already covered by Theorem \ref{thm:local_rig} is actually somewhat subtle.   Bartels--L\"uck--Weinberger \cite{BLW} proved that a torsion-free hyperbolic group with sphere boundary is the fundamental group of a closed, {\em aspherical} manifold, provided that the boundary has dimension at least 5.   However, whether this manifold can be taken to have a Riemannian metric of negative (or even nonpositive) curvature is a separate question.  One could also consider the case where the metric is not assumed Riemannian, but only locally CAT(-1), which is again potentially a separate case; in fact Davis--Januszkiewicz--Lafont \cite{DJL} ask whether there is any example of a smooth, locally CAT(-1) manifold $M$ such that $\partial_\infty \tilde{M} \cong S^{n-1}$ but does not support any Riemannian metric of nonpositive sectional curvature.  To our knowledge this question has not yet been answered.  
Given the subtleties of such metric issues, the spirit of Problem \ref{prob:stable_group} is really to ask for a coarse geometric proof of Theorem \ref{thm:local_rig}, to the extent that this is possible.

There are two other natural directions in which one could attempt to generalize topological stability, the first being a version for closed manifolds with boundary.  

\begin{problem} \label{prob:relative}
Formulate a relative version of Theorem \ref{thm:local_rig} for compact negatively curved manifolds with geodesic boundary, or for finite-volume manifolds of strict negative curvature.  
\end{problem} 

Much more ambitiously, one could attempt a rephrasing of Problem \ref{prob:stable_group} for appropriate classes of relatively hyperbolic groups.  
The basic motivating example for the problem is Thurston's result that the deformation space of hyperbolic structures on a hyperbolic 3-manifold with torus boundary (equivalently, the space of representations of its fundamental group into $\PSL_2(\C) \subset \Homeo(S^2)$, up to $\PSL_2(\C)$-conjugacy) 
has complex dimension equal to the number of boundary components; fixing the structure on the boundary fixes the conjugacy class of the representation.  Problem \ref{prob:relative} asks for a $C^0$ analog of this in a more general setting.   

Finally, it is natural to ask whether any existing techniques can be used to improve the regularity of the (semi)-conjugacy between representations, given higher regularity of the representations.  If, in the context of Theorem \ref{thm:local_rig}, one knows that both $\rho$ and a $C^0$-close representation $\rho'$ are of class $C^k$, for some $k>0$, does it follow that they are in fact conjugate, and if so, conjugate by a $C^k$ diffeomorphism?  
Many existing local rigidity results for group actions use the presence of hyperbolic elements to improve the regularity of a conjugacy (see for instance the foundational work of Katok and Lewis \cite{KL}, as well as Fisher and Margulis \cite{FM}, and Ghys' differentiable rigidity for surface group actions on the circle \cite{Ghys93}).  It is quite possible that some such strategy would directly apply to our case, but we have not pursued this issue.

\subsection{On skew-Anosov flows and slitherings} 

Perhaps the most obvious question arising from this work is the following (folkloric) variant of Christy's question.  
\begin{question}
Does there exist a closed hyperbolic 3-manifold that supports infinitely many inequivalent (skew) Anosov flows? 
\end{question} 
It is our impression that the answer is generally believed to be negative.  The question remains open, and we do not consider our construction of inequivalent flows via surgery to provide any evidence in either direction.  
\medskip

Following the work of Foulon and Hasselblatt \cite{Foulon_Hasselblatt}, the flows that we construct in the proof of Theorem \ref{thm:InequivalentAnosov} are {\em contact Anosov}, meaning that $E^{uu} \oplus E^{ss}$ defines a contact structure with this flow the associated Reeb flow.   
Contact Anosov flows obtained by integral Dehn surgery are studied further in recent work of Foulon, Hasselblatt, and Vaugon, where they ask 
specifically whether this construction can lead to different contact structures on the same manifold.   (See discussion and remarks in \cite{FHV}, Section 2.2.)   In this context one can ask the following. 
\begin{question}
Does there exist a hyperbolic $3$-manifold carrying N distinct contact structures whose Reeb flows are Anosov?
\end{question}
\noindent The examples we produce seem to be natural candidates for this, and hence also for an answer to Foulon--Hasselblatt--Vaugon.

\bibliographystyle{plain}

\bibliography{biblio}

\end{document}